\documentclass{amsart}
\usepackage{amsmath,amsthm,amssymb}
\usepackage{graphicx}

\newtheorem{theorem}{Theorem}[section]
\newtheorem{proposition}[theorem]{Proposition}
\newtheorem{lemma}[theorem]{Lemma}
\newtheorem{corollary}[theorem]{Corollary}

\theoremstyle{definition}
\newtheorem{definition}[theorem]{Definition}
\newtheorem{remark}[theorem]{Remark}
\newtheorem{problem}[theorem]{Problem}

\newtheorem{step}{Step}

\theoremstyle{theorem}

\title[Partial twists and exotic Stein fillings]{Partial twists and exotic Stein fillings}
 
\author[Kouichi Yasui]{Kouichi Yasui}
\thanks{The author was partially supported by JSPS KAKENHI Grant Number 25800048.}
\date{June 25, 2014}

\address{Department~of~Mathematics, Graduate School~of~Science, Hiroshima~University, 1-3-1 Kagamiyama, Higashi-Hiroshima, 739-8526, Japan}
\email{kyasui@hiroshima-u.ac.jp}
\begin{document}

\begin{abstract}
We give an algorithm which produces infinitely many pairwise exotic Stein fillings of the same contact 3-manifolds, applying positive allowable Lefschetz fibrations over the disk. 
As a corollary, for a large class of Stein fillings, we realize the topological invariants (i.e.\ fundamental group, homology group, homology group of the boundary, and intersection form) of each filling as those of infinitely many pairwise exotic Stein fillings. 
Furthermore, applying the algorithm, we produce various contact 3-manifolds of support genus one each of which admits infinitely many pairwise exotic Stein fillings. 
\end{abstract}

\maketitle

\section{Introduction}\label{sec:intro}A fundamental problem in 4-dimensional topology is to find all exotic (i.e.\ homeomorphic but non-diffeomorphic) smooth structures on 4-manifolds. Here we study smooth structures of compact Stein 4-manifolds, since they have many useful properties, and they can be effectively used for surgery constructions of closed exotic symplectic 4-manifolds. Note that various contact 3-manifolds admit unique Stein fillings (i.e.\ compact Stein 4-manifolds) up to diffeomorphism. Therefore Stein fillings of the same contact structures are particularly interesting.  
 The main purpose of this paper is to introduce an algorithm which produces infinitely many pairwise exotic Stein fillings of the same contact 3-manifolds. Furthermore, we determine the support genera of boundary contact 3-manifolds. 

\subsection{Exotic Stein fillings} The first examples of exotic Stein fillings were constructed by Akhmedov-Etnyre-Mark-Smith in \cite{AEMS}. They found infinitely many contact 3-manifolds each of which admits infinitely many pairwise exotic simply connected Stein fillings. Akhmedov-Ozbagci (\cite{AkhOz1}, \cite{AkhOz2}) extended these examples regarding the fundamental groups and the boundary contact 3-manifolds. 
Akbulut and the author \cite{AY6} constructed exotic Stein fillings with small second Betti number ($b_2=2$) using a different method. However, in terms of topological invariants of 4-manifolds, known exotic Stein fillings are relatively few. See also \cite{AY2} and \cite{AY5} for exotic Stein fillings of possibly pairwise distinct contact structures. 

Let us recall that every Stein filling admits a positive allowable Lefschetz fibration over the disk with a bounded fiber surface (PALF), and that the boundary contact structure is compatible with the induced open book on the boundary of the PALF (\cite{LoP}, \cite{AO1}, \cite{Pl}). The converse statement also holds. 

In this paper, we give an algorithm which alter any given PALF $X$ with a certain condition into infinitely many PALF's satisfying the following: (1) they are pairwise exotic; (2) the induced open books on their boundary are pairwise isomorphic; (3) the fundamental group and the homology group of each PALF are isomorphic to those of $X$; (4) For some positive integer $k$, all of these PALF's can be smoothly embedded into the same manifold $X\#k\overline{\mathbb{C}{P}^2}$; (5) All of these PALF's can be embedded into the same PALF, which has one more singular fiber than these fillings, as sub-PALF's. 

Due to the conditions (1)--(3), these PALF's are infinitely many pairwise exotic Stein fillings of the same contact 3-manifold, and  they partly share topological invariants of the given $X$. Furthermore, they have interesting properties. The condition (4) says that they partly share smooth properties of the given $X$, since the blow-up operation preserves those of the original (symplectic) 4-manifold as is well-known. The condition (5) implies that they become pairwise diffeomorphic by attaching just one Stein 2-handle to each filling, though the resulting manifold still admits a Stein structure. 
We note that our algorithm may change the homeomorphism type of the boundary $\partial X$. 
This algorithm, which we introduce in Section~\ref{section:main algorithm}, yields our main result Theorem~\ref{sec:algorithm:mainthm}. See also Subsection~\ref{sec:intro:subsec:on the algorithm}.

We state corollaries of our algorithm. By a \textit{$2$-handlebody}, we mean a connected handlebody obtained from the 0-handle by attaching 1- and 2-handles. It is known that any Stein filling admits a decomposition into a 4-dimensional compact 2-handlebody. 
Our algorithm implies the theorem below, which says that, for a large class of 4-dimensional 2-handlebodies, topological invariants  of each 2-handlebody are realized as those of infinitely many pairwise exotic Stein fillings. 
\begin{theorem}\label{intro:thm:most2-handlebody}Let $X$ be a compact oriented $4$-dimensional $2$-handlebody, and let $Z$ be either $X\#S^2\times S^2$ or $X\#\mathbb{C}{P}^2\#\overline{\mathbb{C}{P}^2}$. Then there exist infinitely many pairwise homeomorphic but non-diffeomorphic Stein fillings of the same contact $3$-manifold such that the fundamental group, the homology group, the homology group of the boundary, and the intersection form of each filling are isomorphic to those of $Z$. Furthermore, for some positive integer $k$, all of these fillings can be smoothly embedded into the same manifold $Z\#k\overline{\mathbb{C}{P}^2}$. 
\end{theorem}


Theorem~\ref{intro:thm:most2-handlebody} immediately gives the corollary below, which was earlier proved by Akhmedov-Ozbagci~\cite{AkhOz2}  using a different method (They furthermore proved that the boundary contact 3-manifolds can be chosen as Seifert fibered singularity links.).

\begin{corollary}[Akhmedov-Ozbagci~\cite{AkhOz2}]For any finitely presented group $G$, there exist infinitely many pairwise homeomorphic but non-diffeomorphic Stein fillings of the same contact $3$-manifold such that the fundamental group of each filling is isomorphic to $G$.
\end{corollary}

Theorem~\ref{intro:thm:most2-handlebody} also gives the corollary below. Namely, for a large class of bilinear forms, we can realize a bilinear form as the intersection form of infinitely many pairwise exotic simply connected Stein fillings. 
\begin{corollary}Let $Q$ and $I$ be any integral symmetric bilinear forms over free modules, and assume that $I$ is indefinite, unimodular, and of rank $2$. Then there exist infinitely many pairwise homeomorphic but non-diffeomorphic simply connected Stein fillings of the same contact $3$-manifold such that the intersection form of each filling is isomorphic to the direct sum $Q\oplus I$.
\end{corollary}

\subsection{Support genera and infinitely many Stein fillings}We next consider support genera which are important invariants of contact 3-manifolds. The support genus of a contact 3-manifold is the minimal page genus of its compatible open book (\cite{EtOz2}). Recently, much attention has been paid to how to characterize contact 3-manifolds which admit finitely many versus infinitely many Stein fillings up to diffeomorphism, and support genera seem to be strongly related to this problem. For example, support genera restrict the homeomorphism types of Stein fillings (e.g.\ \cite{Et3}). Moreover, in the case of support genus zero, various Stein fillable contact 3-manifolds admit only finitely many Stein fillings (e.g.\ \cite{Sch}, \cite{PlV}, \cite{Sta}, \cite{KaL}, \cite{Ka}). In the case of support genus one, the finiteness (uniqueness) also holds for the Stein fillable contact structure on $T^3$ (\cite{We}). 

It is thus natural to ask support genera of contact 3-manifolds admitting infinitely many (not necessarily exotic) Stein fillings. However, the support genera of no such contact 3-manifolds have been determined, though there are many examples of contact 3-manifolds with infinitely many Stein fillings (\cite{Sm}, \cite{OS2}, \cite{AEMS}, \cite{AkhOz1}, \cite{BaV1}, \cite{AY6}, \cite{BaV2}, \cite{AkhOz2}, \cite{DKP}). 
We note that many of such examples naturally admit open books with page genera greater than one, and that there is no known method to show that the support genus of a contact 3-manifold is greater than one. 

In this paper, we produce various contact 3-manifolds of support genus one each of which admits infinitely many pairwise exotic Stein fillings. Indeed, in the case where a given PALF is of genus one, the aforementioned algorithm can produce infinitely many pairwise exotic PALF's of genus one whose boundary contact structure is of support genus one (see Section~\ref{section:main algorithm}). Therefore we obtain vast such examples. This result contrasts sharply with the situation for genus-one \textit{closed}  Lefschetz fibrations (i.e.\ those with the closed total spaces) over $S^2$, where pairwise exotic 4-manifolds do not exist (cf.\ \cite{GS}).    

Among these examples, we here state examples with small second Betti numbers. We hope these small concrete examples of exotic PALF's, which we call Stein nuclei, become useful building blocks for constructing various exotic 4-manifolds.

\begin{theorem}\label{sec:intro:support} There exist infinitely many pairwise non-homeomorphic contact $3$-manifolds of support genus one each of which admits infinitely many pairwise homeomorphic but non-diffeomorphic simply connected Stein fillings with $b_2=2$. Furthermore, each of these $3$-manifolds is a hyperbolic $($hence irreducible$)$ homology $3$-sphere. 
\end{theorem}

We note that some of these fillings are diffeomorphic to the exotic Stein handlebodies obtained in \cite{AY5}. Modifying the algorithm, we also obtain such contact 3-manifolds admitting \textit{non-homeomorphic} Stein fillings with smaller $b_2$. 

\begin{theorem}\label{sec:intro:thm:non-homeo_simple} There exist infinitely many pairwise non-homeomorphic contact $3$-manifolds of support genus one each of which admits infinitely many pairwise non-homeomorphic Stein fillings with $b_1=0$ and $b_2=1$. Furthermore, each of these $3$-manifolds is irreducible and toroidal. 
\end{theorem}

We give more non-homeomorphic examples in the sequel \cite{Y7}. Combining the above examples, we give more exotic examples. The theorem below says that there exists a contact 3-manifold of support genus one, which has infinitely many pairwise non-homeomorphic Stein fillings and admits infinitely many pairwise exotic Stein fillings in each of these homeomorphism types. We remark that similar examples were constructed in \cite{AkhOz1}, though their support genera have not been determined. 

\begin{theorem}\label{sec:intro:thm:non-homeo}There exist infinitely many pairwise non-homeomorphic contact $3$-manifolds of support genus one each of which admits infinitely many Stein fillings $Z_{i,j}$'s $(i,j\in \mathbb{N})$ satisfying the following: for each fixed  $j\in \mathbb{N}$, infinitely many Stein fillings $Z_{i,j}$'s $(i\in \mathbb{N})$ are pairwise homeomorphic but non-diffeomorphic, and for each fixed  $i\in \mathbb{N}$, infinitely many Stein fillings $Z_{i,j}$'s $(j\in \mathbb{N})$ are pairwise non-homeomorphic. 
\end{theorem}


\subsection{On the algorithm}\label{sec:intro:subsec:on the algorithm} Here we briefly explain our algorithm. We require that a given PALF $X$ has a certain subsequence of vanishing cycles. We alter $X$ by modifying the fiber surface and the sequence of vanishing cycles. These simple modifications, which we call $R$-modifications, are related to stabilizations of Legendrian knots via open books. We then obtain infinitely many pairwise exotic PALF's by applying certain monodromy substitutions to the altered PALF along the subsequence. We distinguish smooth structures by using the adjunction inequality, the relative genus function in \cite{Y5}, and handlebody structures of PALF's. 

In order to obtain the algorithm, we realize logarithmic transformations as monodromy substitutions, using the Stein fillable open book on $T^3$ obtained by Van Horn-Morris~\cite{VHM}.  During the preparation of this paper, it turned out that our substitutions are special cases of partial twists defined by Auroux (\cite{Auroux1}, \cite{Auroux2}), though our substitutions (relations) have not been found. 
He also proved that partial twists correspond to logarithmic transformations using a different argument. We would like to emphasize that partial twists (including ours) do not always produce non-diffeomorphic PALF's. Indeed, regarding our algorithm, they do not change the isomorphism (hence diffeomorphism) type of a given PALF $X$ before applying the $R$-modifications (see Remark~\ref{sec:algorithm:remark:thm}). 

Our construction develops the ideas used for the Stein handlebody constructions of Akbulut and the author (\cite{AY5}, \cite{AY6}) and the nucleus constructions of the author (\cite{Y5}). As new ingredients, we use PALF's, $R$-modifications and partial twists in stead of Stein handlebodies. We note that our construction is completely different from those in \cite{AEMS}, \cite{AkhOz1} and \cite{AkhOz2}, where exotic Stein fillings were constructed from exotic \textit{closed} Lefschetz fibrations by cutting out their sections.  

\begin{remark}It is a particularly interesting problem to modify our algorithm so that it produces infinitely many pairwise exotic \textit{closed} Lefschetz fibrations. Since we can embed every PALF into a closed Lefschetz fibration (\cite{AO3}), it is natural to ask whether a PALF constructed by our algorithm can be embedded into a closed Lefschetz fibration, so that partial twists produce infinitely many pairwise exotic closed Lefschetz fibrations. We hope to return to these problems in a future work. 
\end{remark}

This paper is organized as follows. In Section~\ref{sec:PALF}, we summarize the facts about PALF's and Stein fillings. We also fix notations of this paper. In Section~\ref{section:main algorithm}, we state the main result Theorem~\ref{sec:algorithm:mainthm}, and introduce the $R$-modification operation and the algorithm. 
In Section~\ref{sec:partial}, we interpret logarithmic transformations into monodromy substitutions and discuss their relations to Auroux's partial twists. In Section~\ref{sec:rotation}, we study how $R$-modifications and Dehn twists change rotation numbers of curves.  In Section~\ref{sec:proof of main}, we prove the main result. In Section~\ref{sec:ex}, we give many examples demonstrating the algorithm. In particular, we study Stein nuclei and prove the theorems stated in this section. 
\medskip\\
\textbf{Acknowledgements}. The author would like to thank Kenta Hayano, Cagri Karakurt and Burak Ozbagci for their useful comments. A special thanks goes to Yuichi Yamada for kindly pointing out Proposition~\ref{sec:ex:prop:non-homeo}.(5), which improved Theorem~\ref{sec:intro:thm:non-homeo_simple}.

\section{Positive allowable Lefschetz fibrations}\label{sec:PALF}
In this section, we fix the notations of this paper, and briefly recall basics of positive allowable Lefschetz fibrations and their relations to open books, Stein structures, and contact structures.  We refer the reader to \cite{OS1} and the references therein for more details. 
For Lefschetz fibrations, we mainly discuss the case where the base is the disk $D^2$, and the fiber is a surface with non-empty boundary.  

\subsection{Notations}
Throughout this paper, we use the following notations. For two groups $G_1, G_2$, we denote $G_1\cong G_2$ if $G_1$ is isomorphic to $G_2$. For oriented curves $C, D$ in an oriented surface, we denote the algebraic intersection number of $C$ and $D$ by $Q(C, D)$. For an oriented surface $F$, we denote the mapping class group of $F$ by $\textnormal{Aut}(F,\partial F)$. Namely $\textnormal{Aut}(F,\partial F)$ is the group of isotopy classes of orientation preserving self-diffeomorphisms of $F$ which fix the boundary pointwise. By a simple closed curve in a surface, we mean a simple closed curve in the interior of the surface. For a simple closed curve $C$ in a surface, we denote by $t_C$ the right handed Dehn twist along $C$. We use the functional notation for compositions of Dehn twists. 
For simple closed curves $C_1,C_2,\cdots, C_n$ in a surface, we denote the composition $t_{C_n}\circ t_{C_{n-1}}\circ \dots \circ t_{C_1}$ by $(C_1,C_2,\cdots,C_n)$. Beware of the differences of the order. Unless otherwise stated, we do not distinguish an isotopy class of curves and that of diffeomorphisms from their respective representatives. 

\subsection{Monodromy factorizations}In the rest of this section, let $F$ be a compact connected oriented surface with non-empty (possibly disconnected) boundary. 

Let $Z$ be a compact connected oriented smooth 4-manifold with non-empty connected boundary. We call a smooth map $f: Z\to D^2$ a positive Lefschetz fibration over $D^2$ with fiber $F$, if it satisfies the following conditions. 

\begin{itemize}
 \item $f$ has finitely many critical values $b_1,b_2,\dots,b_n$ in the interior of $D^2$, and $f$ is a smooth fiber bundle over $D^2-\{b_1,b_2,\dots,b_n\}$ with fiber $F$. 
 \item For each $i$, there is a unique critical point $p_i$ in the singular fiber $f^{-1}(b_i)$. Furthermore,  around each $p_i$ and $b_i$, $f$ is locally given by $f(z_1, z_2)=z_1^2+z_2^2$ with respect to the local complex coordinate charts compatible with the orientations of $Z$ and $D^2$. 
\end{itemize}
Two positive Lefschetz fibrations $f:Z\to D^2$ and $f':Z'\to D^2$ are said to be isomorphic, if there exist orientation-preserving diffeomorphisms $\widetilde{g}:Z\to Z'$ and $g: D^2\to D^2$ satisfying $g\circ f=f'\circ \widetilde{g}$. 

Each singular fiber of a positive Lefschetz fibration $f$ is obtained from a regular fiber $F$ by collapsing a simple closed curve, which is called a vanishing cycle. A positive Lefschetz fibration $f:Z\to D^2$ is called a \textit{positive allowable Lefschetz fibration} (PALF for short), if every vanishing cycle of the fibration is homologically non-trivial in $F$. We also call the total space $Z$ a PALF abusing the terminology. 

The monodromy of a PALF is given by $t_{C_n}\circ t_{C_{n-1}}\circ \dots \circ t_{C_1}$, where $C_1,C_2,\cdots,C_n$ are vanishing cycles of $f$ in $F$. A PALF is determined by a factorization of its monodromy into a composition of right handed Dehn twists (along vanishing cycles). Conversely, for any collection of homologically non-trivial (ordered) simple closed curves $C_1,C_2, \dots, C_n$ in $F$, there exists a PALF with fiber surface $F$ such that its monodromy factorization is $t_{C_n}\circ t_{C_{n-1}}\circ \dots \circ t_{C_1}$. 

The following three operations for monodromy factorizations preserve the isomorphism class of a PALF. The first is to change a monodromy factorization (and the monodromy itself) using cyclic permutations. The second is to replace the original monodromy factorization $(C_1,C_2,\dots,C_n)$ with $(\psi(C_1), \psi(C_2), \dots, \psi(C_{n}))$  for any $\psi\in \textnormal{Aut}(F,\partial F)$. This operation is called a simultaneous conjugation, since $t_{\psi(C_i)}=\psi \circ t_{C_i}\circ \psi^{-1}$. The third is to replace the monodromy factorization $(C_1, \dots,C_i, C_{i+1}, C_{i+2},\dots, C_n)$ with either of the following two factorizations 
\begin{equation*}
(C_1, \dots, C_{i+1}, t_{C_{i+1}}^{-1}(C_{i}), C_{i+2},\dots, C_n), \quad
(C_1, \dots, t_{C_i}(C_{i+1}), C_i, C_{i+2},\dots, C_n).
\end{equation*}
This operation is called an elementary transformation. Conversely, monodromy factorizations of isomorphic PALF's are related to each other by these three operations. Note that the third operation keeps the original monodromy, though the others may change.


\subsection{Open books and contact structures}
An (abstract) open book is a pair $(F, \varphi)$, where $F$ is a compact connected oriented surface with non-empty boundary, and $\varphi$ is an element of $\textnormal{Aut}(F,\partial F)$. Two open books $(F,\varphi)$ and $(F',\varphi')$ are said to be isomorphic, if there exists an orientation-preserving diffeomorphism $f:F\to F'$ such that $f\circ \varphi=\varphi'\circ f$. An open book $(F,\varphi)$ gives a closed oriented 3-manifold $M_{(F, \varphi)}$ by attaching $S^1\times D^2$'s to the mapping torus $[0,1]\times F/((1,x)\sim (0,\varphi(x)))$, where we identify $S^1$-factors of $S^1\times D^2$'s with the boundary $\partial F$. A surface $\{pt.\}\times F\subset I\times F$ is called a page of the open book $(F,\varphi)$. An open book $(F,\varphi)$ induces the compatible contact structure on $M_{(F, \varphi)}$. For a contact structure $\xi$ on a 3-manifold $M$, the support genus $\textnormal{sg}(\xi)$ of $\xi$ is the minimal genus of a page of an open book on $M$ whose induced contact structure is isomorphic to $\xi$ (\cite{EtOz2}). 
A PALF induces an open book $(F, \varphi)$ on the boundary $\partial Z$ of the total space $Z$, where the monodromy $\varphi\in \textnormal{Aut}(F,\partial F)$ is given by $\varphi=t_{C_n}\circ t_{C_{n-1}}\circ \dots \circ t_{C_1}$. 

\subsection{Stein structures on PALF's}\label{sec:Lef:subsec:Stein}
The total space $Z$ of a PALF naturally induces a handle decomposition, and the decomposition induces a Stein structure as follows (\cite{AO1}). The 4-manifold $Z$ is obtained from $F\times D^2$ by attaching 2-handles along vanishing cycles $C_1, C_2, \dots, C_n$ in pairwise distinct pages of the open book $(F, id)$ on $\partial(F\times D^2)$. Since $F\times D^2$ is a boundary sum of $S^1\times D^3$'s, $F\times D^2$ admits a Stein structure, and the induced contact structure on $\partial(F\times D^2)$ is compatible with the open book $(F,id)$. Due to the Legendrian realization principle, we may assume that each $C_i$ is a Legendrian knot in a page of the open book, and that the contact framing coincides with the surface framing. 
Therefore, the Stein structure on $F\times D^2$ extends to $Z$ according to Eliashberg's theorem in \cite{E1}. The induced contact structure $\xi$ on $\partial Z$ is compatible with the open book $(F,\varphi)$ on $\partial Z$ induced from the PALF structure. Consequently, $Z$ is a Stein filling of the contact 3-manifold $(\partial Z,\xi)$. 
\subsection{Handlebody diagram and the first Chern class}\label{subsec:Chern}
Etnyre-Ozbagci~\cite{EtOz2} (See also Proposition 2.3 in \cite{G1}) gave a formula of the first Chern class $c_1(Z)$ of the Stein structure on a PALF $Z$, via the handlebody structure of $Z$ explained in the above subsection. To compute $c_1(Z)$ easily, we use a special handlebody diagram of $Z$ obtained as follows, similarly to \cite{EtOz2}. 

We start with a handlebody diagram of the fiber surface $F$. For a positive integer $k$, choose pairwise disjoint 2-disks $D_1, D_2, \dots, D_k$ in $\mathbb{R}^2$. Here we require that each $D_i$ is a (rounded) rectangle and that any segment of the boundary of the rectangle is parallel to either the $x$- or $y$-axis of $\mathbb{R}^2$. Using bands in $\mathbb{R}^2$, take a boundary sum of $D_1, D_2, \dots, D_k$ so that the resulting surface is diffeomorphic to a disk. This is our 0-handle of $F$ embedded in $\mathbb{R}^2$. We attach each 1-handle of $F$ to the 0-handle either vertically or horizontally. Namely, the two end points of the attaching sphere have the same value with respect to either the $x$- or $y$-axis. The resulting handlebody diagram (i.e.\ the $0$-handle in $\mathbb{R}^2$ with the attaching regions of the 1-handles specified) gives the oriented surface $F$, where we use the orientation induced from the standard orientation of $\mathbb{R}^2$. For such an example, see Figure~\ref{fig:example_handle}, where the 0-handle consists of three (rounded) rectangles, and the three vertical and two horizontal 1-handles are attached to the 0-handle along the red regions. 

\begin{figure}[h!]
\begin{center}
\includegraphics[width=2.6in]{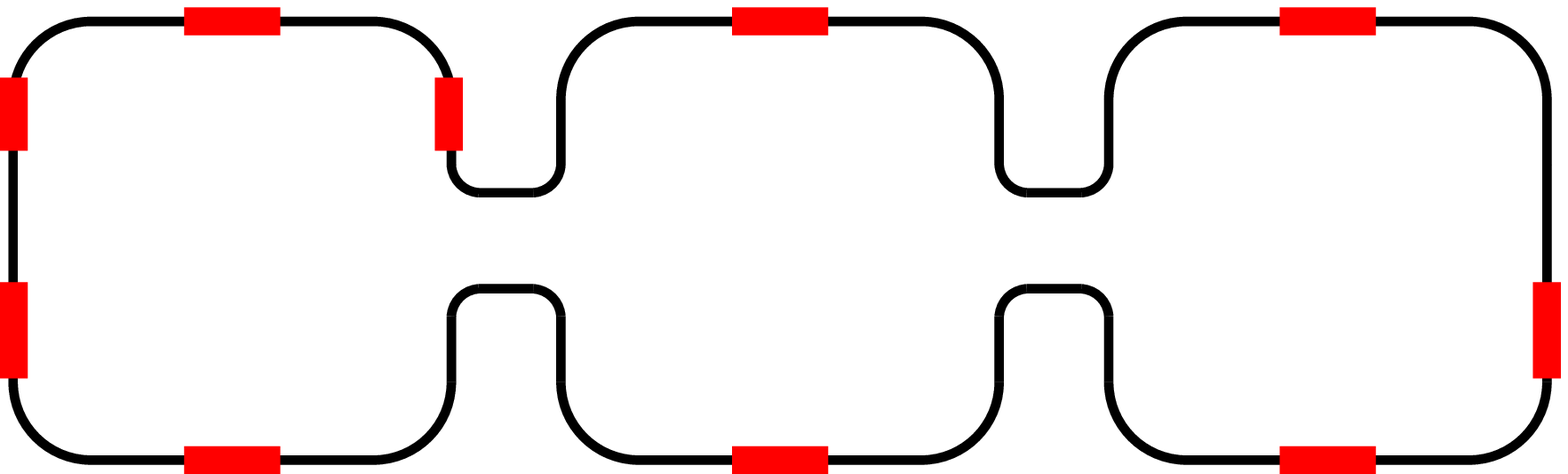}
\caption{Example}
\label{fig:example_handle}
\end{center}
\end{figure}

The handlebody diagram of $F$ leads to a handle diagram of $F\times D^2$ drawn in $\mathbb{R}^3$. We choose points $p_1,p_2,\dots, p_n$ in $\partial D^2$ so that $p_i$ moves to the positive direction of $\partial D^2$ if $i$ increases. We orient vanishing cycles $C_1, C_2, \dots, C_n$ to define the rotation numbers. For each $i$, attach a 2-handle to $F\times D^2$ along the vanishing cycle $C_i$ in the page $F\times p_i$, where the framing of the attaching circle is $-1$ relative to its surface framing. We orient each attaching circle so that its orientation coincides with that of the corresponding vanishing cycle. The resulting picture is a desired handlebody diagram of the PALF $Z$ whose monodromy factorization is $(C_1,C_2,\dots, C_n)$. 

We fix a trivialization of the tangent bundle of the fiber $F$ as follows. We restrict the standard trivialization of the tangent bundle of $\mathbb{R}^2$ to that of the 0-handle of $F$, and then extend the trivialization to the 1-handles of $F$. 

Here, for an oriented simple closed curve $C$ in $F$, we define the rotation number $r(C)$ of $C$ as the winding number of the tangent vector field to $C$ with respect to the above trivialization of the tangent bundle of $F$. Note that any isotopy of a curve does not change its rotation number. The well-known lemma below simplifies calculations of rotation numbers.

\begin{lemma}\label{lem:comp:winding}The winding number of an oriented closed $($possibly non-simple$)$ curve in $\mathbb{R}^2-\{0\}$ equals the algebraic intersection number of any fixed half-line and the curve. 
\end{lemma}

Using the above handlebody diagram of the PALF $Z$, we can compute $c_1(Z)$ as follows. 
\begin{proposition}[\cite{EtOz2}. See also \cite{G1}]\label{sec:PALF:Chern} The first Chern class $c_1(Z)\in H^2(Z;\mathbb{Z})$ of the Stein structure on $Z$ induced from the PALF structure is represented by a cocycle whose value on each $2$-handle corresponding to the vanishing cycle $C_i$ is the rotation number $r(C_i)$. Here we regard $2$-handles as a basis of the $2$-chain group. 
\end{proposition}


As pointed out in \cite{AM}, due to the embedding theorem of Stein manifolds in \cite{LM1} together with the well-known adjunction inequality for closed 4-manifolds (cf.\ \cite{GS}),  we obtain the following adjunction inequality for PALF's (Stein manifolds). Note that this version of the adjunction inequality also holds in the genus zero case (cf.\ \cite{OS1}), unlike the version for closed 4-manifolds. 
\begin{theorem}[\cite{AM}. cf.\ \cite{OS1}]Let $\Sigma$ be a smoothly embedded closed connected oriented surface of genus $g\geq 0$ in a PALF $Z$, and let $[\Sigma]$ be the second homology class of $Z$ represented by $\Sigma$. If $[\Sigma]\neq 0$, then the following adjunction inequality holds. 
\begin{equation*}
\left| \langle c_1(Z), [\Sigma] \rangle \right|+ [\Sigma]\cdot [\Sigma]\leq 2g-2. 
\end{equation*}
\end{theorem}

\section{Main result and algorithm}\label{section:main algorithm}
In this section, we state our main result and a corollary and give an algorithm which produces exotic Stein fillings via PALF's. We also introduce the $R$-modification operation.

\subsection{The main result}\label{subsec:S} To state the main theorem, we need to introduce some definitions. Let $\mathbb{S}$ be the compact connected oriented surface of genus one with three boundary components in Figure~\ref{torus_3-holes_v2}, and let $\alpha_1, \alpha_2, \alpha_3, \beta, \delta_1, \delta_2, \delta_3$ and $\tau_1,\tau_2,\tau_3$ be the oriented simple closed curves and the simple proper arcs in $\mathbb{S}$ shown in  Figure~\ref{torus_3-holes_v2}. The orientation of $\mathbb{S}$ is the one satisfying $Q(\alpha_1,\beta)=1$. Note that $\delta_1,\delta_2,\delta_3$ are boundary parallel curves. For integers $i,j$, we define the simple closed curves $\gamma_{i}, \gamma_{i}^{(j)}, \beta^{(j)}$ in $\mathbb{S}$ by
\begin{equation*}
\gamma_{i}=(t_{\alpha_3}\circ t_{\alpha_2}\circ t_{\alpha_1})^i(\beta), \quad \gamma_{i}^{(j)}=t_{\alpha_1}^j(\gamma_{i}), \quad \beta^{(j)}=t_{\alpha_1}^j(\beta). 
\end{equation*}
We orient $\gamma_{i}, \gamma_{i}^{(j)}, \beta^{(j)}$ by extending the orientations of $\beta, \gamma_{i}, \beta$, respectively. 
\begin{figure}[h!]
\begin{center}
\includegraphics[width=1.7in]{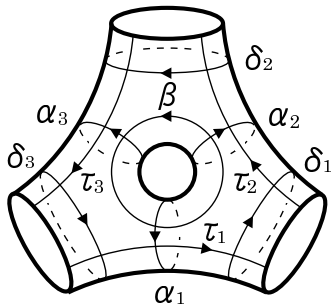}
\caption{The compact surface $\mathbb{S}$ of genus one with three holes}
\label{torus_3-holes_v2}
\end{center}
\end{figure}

We are ready to state our main theorem. Let $X$ be an arbitrary PALF satisfying the following conditions. 
\begin{itemize}
 \item The regular fiber $\Sigma$ of $X$ is a compact oriented surface obtained from $\mathbb{S}$ by attaching 1-handles to the boundary $\partial \mathbb{S}$. We allow the case $\Sigma=\mathbb{S}$. 
 \item A monodromy factorization of $X$ is $(\gamma_{1}, \beta, \gamma_{-1}, C_1, C_2, \dots, C_n)$, where $n$ is an arbitrary positive integer, and $C_1,C_2,\dots, C_n$ are homologically non-trivial arbitrary simple closed curves in the interior of $\Sigma$.
 \item $\alpha_1,\alpha_2, \gamma_{-1}$ are contained in the set $\{C_1, C_2,\dots, C_n\}$.
\end{itemize}
The main theorem of this paper is the following. 
\begin{theorem}\label{sec:algorithm:mainthm}For any PALF $X$ satisfying the above conditions, there exist infinitely many pairwise homeomorphic but non-diffeomorphic Stein fillings of the same contact $3$-manifold such that the fundamental group and the homology group of each filling are isomorphic to those of the given $X$. Furthermore, these fillings satisfy the following conditions. 
\begin{itemize}
 \item For some positive integer $k$, all of these fillings can be smoothly embedded into the same manifold $X\#k\overline{\mathbb{C}{P}^2}$. 
 \item All of these fillings become pairwise diffeomorphic by attaching the same Stein $2$-handle to each filling, namely, by attaching a $2$-handle to each fillings along the same Legendrian knot with contact $-1$-framing. 
 \item If the genus of the fiber $\Sigma$ of $X$ is one, then we may assume that the boundary contact $3$-manifold of these fillings is of support genus one. 
\end{itemize}

\end{theorem}
\begin{remark}While we required that the specific curves $\alpha_1,\alpha_2, \gamma_{-1}$ are contained in the set $\{C_1, C_2,\dots, C_n\}$, we can replace these three curves with more general curves in $\Sigma$. This can be seen from the proof, and we omit the details. 
\end{remark}
In the rest of this section, we give an algorithm which produces these exotic Stein fillings from the given $X$. We also state a corollary. We need to introduce some definitions and operations. 

\subsection{A handle decomposition $\widehat{\mathbb{S}}$ of $\mathbb{S}$}\label{subsec:handlebody of S}
We first fix a handle decomposition of the surface $\mathbb{S}$.  Let $\widehat{\mathbb{S}}$ be the 2-dimensional handlebody in Figure~\ref{fig:S_hat_v2}, where three vertical and one horizontal 1-handles are attached to the $0$-handle in $\mathbb{R}^2$ along the red regions. Though this 0-handle consists of only one rectangle in $\mathbb{R}^2$, if necessary, we may assume that the 0-handle of $\widehat{\mathbb{S}}$ contains arbitrarily many rectangles by taking boundary sums with rectangles in $\mathbb{R}^2$. We orient $\widehat{\mathbb{S}}$ by extending the orientation of the $0$-handle induced from the standard orientation of $\mathbb{R}^2$ to the 1-handles. 

Let $\alpha_1, \alpha_2, \alpha_3,\beta$ be the oriented simple closed curves in $\widehat{\mathbb{S}}$ shown in Figure~\ref{fig:S_hat_v2}, and let $\tau_1, \tau_2, \tau_3,\tau_{\beta}$ be the cocores of the 1-handles which intersect with $\alpha_1, \alpha_2, \alpha_3,\beta$, respectively. We orient each $\tau_i$ and $\tau_{\beta}$ so that $Q(\alpha_i, \tau_i)=+1$ and $Q(\beta, \tau_{\beta})=+1$. Note that there exists a diffeomorphism $\psi_{\mathbb{S}} :\mathbb{S}\to \widehat{\mathbb{S}}$ which sends the oriented curves and arcs $\alpha_1, \alpha_2, \alpha_3, \tau_1,\tau_2,\tau_3,\beta$ in Figure~\ref{torus_3-holes_v2} to those in Figure~\ref{fig:S_hat_v2}. 
\begin{figure}[h!]
\begin{center}
\includegraphics[width=1.8in]{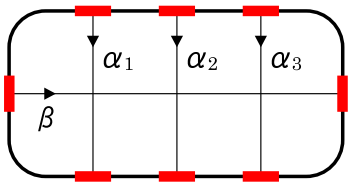}
\caption{The handle decomposition $\widehat{\mathbb{S}}$ of $\mathbb{S}$}
\label{fig:S_hat_v2}
\end{center}
\end{figure}

\subsection{$R$-modification} We next introduce an operation of simple closed curves in surfaces. 
Let $C$ be a simple closed curve in a compact oriented surface $F$ with non-empty (possibly disconnected) boundary. Attach a 1-handle to $F$ so that the resulting surface $F'$ is oriented. Note that the number of boundary components and the genus of $F'$ depend on the choice of the new 1-handle. 
Let $E$ be an arbitrary simple closed curve in $F'$ such that $E$ intersects with the cocore of the new 1-handle geometrically once, and that $E$ does not geometrically intersect with $C$. Let $C'$ be a band connected sum of $C$ and a parallel copy of $E$ in $F'$.

\begin{definition}We say that the above operation is an \textit{$R$-modification} to $C$, that $C'$ (resp.\ $F'$) is the curve (resp.\ the oriented surface) obtained by applying the $R$-modification to $C$, and that $E$ is the \textit{auxiliary curve} of the $R$-modification.
\end{definition}

Figure~\ref{fig:R-modification} describes a simple example of an $R$-modification. 
\begin{figure}[h!]
\begin{center}
\includegraphics[width=3.4in]{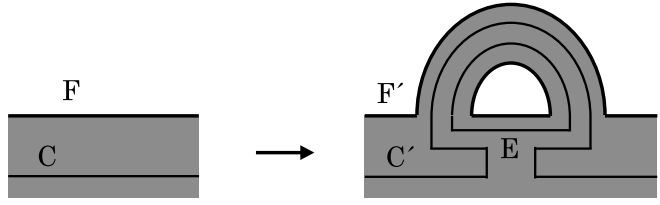}
\caption{An example of an $R$-modification to $C$}
\label{fig:R-modification}
\end{center}
\end{figure}

\begin{remark}$(1)$ As seen from Figure~\ref{fig:R-modification}, we can always apply an $R$-modification to a given curve $C$ so that the genus of the resulting surface $F'$ is equal to the genus of the original surface $F$. 

$(2)$ We introduced $R$-modifications, inspired from the stabilizations of Legendrian knots in \cite{Et3} and \cite{On}. Indeed, due to \cite{Et3} and \cite{On}, special $R$-modifications correspond to stabilizations of Legendrian knots via open books as follows. Assume that $C$ is a homologically non-trivial curve in a page of an open book $(F,\varphi)$. Due to the Legendrian realization principle, we can regard $C$ as a Legendrian knot in the contact 3-manifold compatible with $(F, \varphi)$. Apply an $R$-modification to $C$ as shown in Figure~\ref{fig:R-modification}. Then the resulting curve $C'$, which sits in a page of the stabilized open book $(F', \varphi\circ t_{E})$, realizes a stabilization of the Legendrian knot $C$. 
\end{remark}

\subsection{The algorithm}\label{subsec:algorithm}
Now we are ready to introduce our algorithm. Let $X$ be a PALF satisfying the conditions in Subsection~\ref{subsec:S}. In the following steps, we modify the vanishing cycles and the fiber of the given $X$. 
\begin{step}\label{step:modified PALF} Fix an $n$-tuple $m=(m_1,m_2,\dots, m_n)$ of non-negative integers. For each $1\leq j\leq n$, apply $R$-modifications to $C_j$ $m_j$ times (see Figure~\ref{fig:step_1} for an example of such modifications). Let $C_j(m_j)$ denote the resulting curve obtained from $C_j$, and let $E_k^{j}$ denote the auxiliary curve of the $k$-th $R$-modification to $C_j$. Let $\Sigma^{(m)}$ denote the oriented surface obtained from $\Sigma$ by applying these $m_1, m_2, \dots, m_n$ times $R$-modifications to $C_1,C_2, \dots, C_n$, respectively. For an integer $i$, let $X^{(m)}$, $X^{(m)}_i$ and $\widetilde{X}^{(m)}_i$ be the PALF's with fiber $\Sigma^{(m)}$ whose monodromy factorizations are
\begin{align*}
(\gamma_{1}, \beta, \gamma_{-1}, \, & E_1^{1}, E_2^{1}, \cdots, E_{m_1}^{1}, C_1(m_1), \\
 &E_1^{2}, E_2^{2}, \dots, E_{m_2}^{2}, C_2(m_2), \dots, E_{1}^{n}, E_{2}^{n}, \dots, E_{m_{n}}^{n}, C_n(m_n)),\\
(\gamma_{1}^{(i)}, \beta^{(i)}, \gamma_{-1}^{(i)}, \, & E_1^{1}, E_2^{1}, \cdots, E_{m_1}^{1}, C_1(m_1), \\
 &E_1^{2}, E_2^{2}, \dots, E_{m_2}^{2}, C_2(m_2), \dots, E_{1}^{n}, E_{2}^{n}, \dots, E_{m_{n}}^{n}, C_n(m_n)),\\ 
(\gamma_{1}^{(i)}, \beta^{(i)}, \gamma_{-1}^{(i)}, \, & E_1^{1}, E_2^{1}, \cdots, E_{m_1}^{1}, C_1(m_1), \\
 &E_1^{2}, E_2^{2}, \dots, E_{m_2}^{2}, C_2(m_2), \dots, E_{1}^{n}, E_{2}^{n}, \dots, E_{m_{n}}^{n}, C_n(m_n), \alpha_1), 
 \end{align*}
 respectively.  Clearly, $X^{(0)}=X$ and $X^{(m)}_0=X^{(m)}$. Note that the monodromy factorization of $X^{(m)}$ is obtained from the one of $X$ by replacing each $(C_j)$ with 
\begin{equation*}
(E_1^{j}, E_2^{j}, \cdots, E_{m_j}^{j}, C_j(m_j)). 
\end{equation*}
We equip $X^{(m)}$ and $X^{(m)}_i$ with the Stein structures induced from their respective PALF structures. 
Let $\xi^{(m)}$ and $\xi^{(m)}_i$ be the contact structures on the boundaries $\partial X^{(m)}$ and $\partial X^{(m)}_i$ induced from the Stein structures on $X^{(m)}$ and $X^{(m)}_i$, respectively. This completes Step~\ref{step:modified PALF}. \qed
\end{step}

\begin{figure}[h!]
\begin{center}
\includegraphics[width=4.8in]{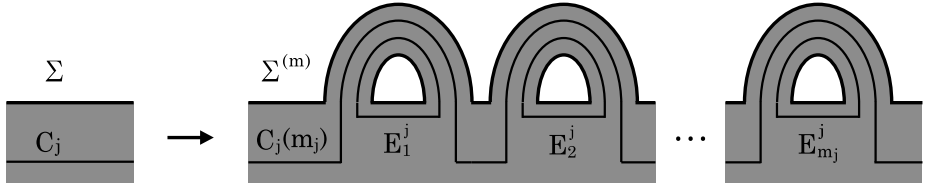}
\caption{An example of Step~\ref{step:modified PALF}}
\label{fig:step_1}
\end{center}
\end{figure}

\begin{remark}$(1)$ The PALF structures (e.g.\ the genus of the fiber $\Sigma^{(m)}$) on $X^{(m)}$ and $X^{(m)}_i$ depend on the choices of $R$-modifications. However, the diffeomorphism types of $X^{(m)}$ and $X^{(m)}_i$ do not depend on the way to apply them. See Lemma~\ref{sec:R:lem:diffeo:modification}. 

$(2)$ Each $X^{(m)}_i$ is obtained from $X^{(m)}$ by using a logarithmic transformation. See Section~\ref{sec:partial}. 
\end{remark}

Though Step~\ref{step:modified PALF} essentially finishes the construction, we need more steps to find an appropriate $n$-tuple $m$. Here we list the properties of $X_i^{(m)}$'s, which are proved in Section~\ref{sec:proof of main}. 
\begin{proposition}\label{sec:algorithm:prop:PALF}
Fix an arbitrary $n$-tuple $m=(m_1,m_2,\dots, m_n)$ of non-negative integers. Then the following hold. 

$(1)$ The fundamental group and the homology group of each $X^{(m)}_i$ $(i\in \mathbb{Z})$ are isomorphic to those of $X$. 

$(2)$ $X_{2i}^{(m)}$'s $(i\in \mathbb{Z})$ are all homeomorphic to $X^{(m)}$. 

$(3)$ $X_{2i-1}^{(m)}$'s $(i\in \mathbb{Z})$ are pairwise homeomorphic. 

$(4)$ The open book on each $\partial X_i^{(m)}$ $(i\in \mathbb{Z})$ induced from the PALF $X_i^{(m)}$ is isomorphic to the one on $\partial X^{(m)}$ induced from the PALF $X^{(m)}$. Consequently, each $(\partial X_i^{(m)},\xi^{(m)}_i)$ is contactomorphic to $(\partial X^{(m)},\xi^{(m)})$. 

$(5)$ The intersection form of each $X^{(m)}_i$ $(i\in \mathbb{Z})$ is indefinite. Consequently, $\textnormal{sg}(\xi^{(m)}_i)\geq 1$. 

$(6)$ Each $X_{i}^{(m)}$ $(i\in \mathbb{Z})$ can be smoothly embedded into $X\#_{j=1}^n m_j\overline{\mathbb{C}{P}^2}$.

$(7)$ The PALF's $\widetilde{X}_{i}^{(m)}$'s $(i\in \mathbb{Z})$ are pairwise isomorphic. Consequently, $X_{i}^{(m)}$'s $(i\in \mathbb{Z})$ are sub-PALF's of the same PALF $\widetilde{X}_{0}^{(m)}$. 
\end{proposition}

We proceed the algorithm. To define the rotation numbers for curves in $\Sigma$, we need a handle decomposition of $\Sigma$ as in Subsection~\ref{subsec:Chern}. 
\begin{step}\label{step:handle} Find a handle decomposition $\widehat{\Sigma}$ of $\Sigma$ satisfying the following (Such a decomposition clearly exists.). 
\begin{itemize}
 \item The handlebody $\widehat{\Sigma}$ is obtained from $\widehat{\mathbb{S}}$ by attaching 1-handles to the boundary of the 0-handle of $\widehat{\mathbb{S}}$.  
 \item Each 1-handle of $\widehat{\Sigma}$ is attached either vertically or horizontally. Namely, two attaching points of each 1-handle have the same coordinate with respect to either the $x$- or $y$-axis of $\mathbb{R}^2$.  
 \item There exists a diffeomorphism $\psi_{\Sigma}:\Sigma\to \widehat{\Sigma}$ whose restriction to $\mathbb{S}$ is the diffeomorphism $\psi_{\mathbb{S}} :\mathbb{S}\to \widehat{\mathbb{S}}$ in Subsection~\ref{subsec:S}. 
\end{itemize}
This completes Step~\ref{step:handle}. \qed
\end{step}

Finally, we choose an appropriate $n$-tuple $m$. Roughly speaking, the conditions below require that each $m_j$ is sufficiently large. 
\begin{step}\label{step:condition:exotic}  Regard curves and arcs in $\Sigma$ as those in  $\widehat{\Sigma}$ via the diffeomorphism $\psi_{\Sigma}$ given by Step~\ref{step:handle}. Orient each simple closed curve $C_j$ $(1\leq j\leq n)$ so that the oriented curves $\alpha_1,\alpha_2, \gamma_{-1}$ are contained in the set  $\{C_1, C_2,\dots, C_n\}$. Trivialize the tangent bundle of the handlebody $\widehat{\Sigma}$ as explained in Subsection~\ref{subsec:Chern}, and  let $r(C_j)$ $(1\leq j\leq n)$ denote the rotation number of $C_j$ with respect to this trivialization. Let $J$ and $J_{\alpha_1}$ be the sets of indices defined by
\begin{align*}
J&=\{j\in \mathbb{N}\mid 1\leq j\leq n, \; \; Q(C_j,\tau_1)-Q(C_j,\tau_3)\neq 0\}, \\
J_{\alpha_1}&=\{j\in J\mid \text{$C_j$ is isotopic to $\alpha_1$ in $\Sigma$,}\\
&\hspace{1in}\text{preserving the orientations of $C_j$ and $\alpha_1$.}\}.
\end{align*}
Find an $n$-tuple $m=(m_1,m_2,\dots, m_n)$ of non-negative integers which satisfies the following inequality for each $1\leq j\leq n$. 
\begin{equation*}
m_j\geq \left\{
\begin{array}{lll}
|r(C_{j})|+|Q(C_j,\tau_1)-Q(C_j,\tau_3)|,   &\text{if $j\in J_{\alpha_1};$}\\
|r(C_{j})|+|Q(C_j,\tau_1)-Q(C_j,\tau_3)|, &\text{if $j\in J-J_{\alpha_1}$, and $m_k\equiv 1$}\\
&\text{$\pmod{2}$ for some $k\in J_{\alpha_1};$}\\
|r(C_{j})|+2|Q(C_j,\tau_1)-Q(C_j,\tau_3)|, &\text{if $j\in J-J_{\alpha_1}$, and $m_k\equiv 0$}\\
&\text{$\pmod{2}$ for any $k\in J_{\alpha_1};$}\\
0, &\text{if $j\not\in J$.}
\end{array}
\right.
\end{equation*}
This completes Step~\ref{step:condition:exotic} and the algorithm. \qed 
\end{step}
\begin{remark}\label{sec:main:rem:step3}(1) The rotation numbers of $\alpha_1, \alpha_2, \alpha_3, \beta, \gamma_{-1},\gamma_{1}$ are as follows.  
\begin{equation*}
r(\alpha_1)=r(\alpha_2)=r(\alpha_3)=r(\beta)=r(\gamma_{-1})=r(\gamma_{1})=0.
\end{equation*}
One can easily check this. See also Lemma~\ref{lem:rotation:dehn}. 

(2) For $j\in J_{\alpha_1}$, the above condition is equivalent to the condition $m_j\geq 1$. 

(3) If $C_j$ is contained in the complement of $\mathbb{S}$ in $\Sigma$, then $m_j=0$ satisfies the above condition. 
\end{remark}

The resulting $X^{(m)}_i$'s satisfy the following, which is proved in Section~\ref{sec:proof of main}. 
\begin{theorem}\label{sec:algorithm:thm:exotic_PALF}
Fix an $n$-tuple $m=(m_1,m_2,\dots, m_n)$ of non-negative integers satisfying the conditions in Step~\ref{step:condition:exotic}. Then the following hold. 

$(1)$ $X_{2i}^{(m)}$'s $(i\in \mathbb{Z})$ are pairwise homeomorphic Stein fillings of the same contact $3$-manifold $(\partial X^{(m)}, \xi^{(m)})$. Moreover, infinitely many of them are pairwise non-diffeomorphic.  

$(2)$ $X_{2i-1}^{(m)}$'s $(i\in \mathbb{Z})$ are pairwise homeomorphic Stein fillings of the same contact $3$-manifold $(\partial X^{(m)}, \xi^{(m)})$. Moreover, infinitely many of them are pairwise non-diffeomorphic.   

$(3)$ The fundamental group and the homology group of each $X^{(m)}_i$ $(i\in \mathbb{Z})$ are isomorphic to those of $X$. 

$(4)$ $\textnormal{sg}(\xi^{(m)})\geq 1$. Consequently, $\textnormal{sg}(\xi^{(m)})=1$, if the genus of the fiber $\Sigma^{(m)}$ of $X^{(m)}$ is one. 

$(5)$ Each $X_{i}^{(m)}$ $(i\in \mathbb{Z})$ can be smoothly embedded into $X\#_{j=1}^n m_j\overline{\mathbb{C}{P}^2}$.

$(6)$ $X_{i}^{(m)}$'s $(i\in \mathbb{Z})$ become pairwise diffeomorphic by attaching a $2$-handle to each $X_{i}^{(m)}$ along the same Legendrian knot in $(\partial X^{(m)}, \xi^{(m)})$ with the contact $-1$-framing. 
\end{theorem}

\begin{remark}\label{sec:algorithm:remark:thm}
Without the conditions of $m$ in Step~\ref{step:condition:exotic}, the above theorem does not always hold. Indeed, Lemma~\ref{sec:partial:lem:condition:unchange} guarantees that each PALF $X^{(m)}_i$ is isomorphic (hence diffeomorphic) to $X^{(m)}$, if some index $j$ satisfies both $m_j=0$ and $C_j=\alpha_1$. In particular, every $X^{(0)}_i$ is isomorphic to the original $X$ for any integer $i$. 
\end{remark}

The above theorem immediately gives the main theorem. 
\begin{proof}[Proof of Theorem~\ref{sec:algorithm:mainthm}]
According to Theorem~\ref{sec:algorithm:thm:exotic_PALF}, the PALF's $X^{(m)}_i$'s give desired exotic Stein fillings, and $\textnormal{sg}(\xi^{(m)})\geq 1$. Suppose that the genus of the fiber $\Sigma$ of $X$ is one. Then we may assume that the genus of the fiber $\Sigma^{(m)}$ of each $X^{(m)}_i$ is one, if necessary by replacing $R$-modifications applied to $C_j$'s. Therefore, we may assume $\textnormal{sg}(\xi^{(m)})=1$. 
\end{proof}

We state a corollary of Theorem~\ref{sec:algorithm:thm:exotic_PALF}, which is proved in Section~\ref{sec:proof of main}. 

\begin{corollary}\label{sec:algorithm:cor:boundary sum}
Fix an $n$-tuple $m=(m_1,m_2,\dots, m_n)$ of non-negative integers satisfying the conditions in Step~\ref{step:condition:exotic}. Then, for any Stein filling $Y$ of a contact $3$-manifold, there exists a contact structure $\zeta^{(m)}$ on $\partial X^{(m)}\#\partial Y$ satisfying the following. 

$(1)$ The boundary connected sums $X_{2i}^{(m)}\natural Y$'s $(i\in \mathbb{Z})$ are pairwise homeomorphic Stein fillings of the same contact $3$-manifold $(\partial X^{(m)}\#\partial Y, \zeta^{(m)})$. Moreover, infinitely many of them are pairwise non-diffeomorphic.  

$(2)$ $X_{2i-1}^{(m)}\natural Y$'s $(i\in \mathbb{Z})$ are pairwise homeomorphic Stein fillings of the same contact $3$-manifold $(\partial X^{(m)}\#\partial Y, \zeta^{(m)})$. Moreover, infinitely many of them are pairwise non-diffeomorphic.   

$(3)$ Each $X_{i}^{(m)}\natural Y$ $(i\in \mathbb{Z})$ can be smoothly embedded into $X\natural Y\#_{j=1}^n m_j\overline{\mathbb{C}{P}^2}$.

$(4)$ $\textnormal{sg}(\zeta^{(m)})\geq 1$. Furthermore, if $Y$ admits a PALF structure with genus zero fiber surface, and the fiber $\Sigma^{(m)}$ of $X^{(m)}$ is of genus one, then we may assume $\textnormal{sg}(\zeta^{(m)})=1$. 
\end{corollary}
This corollary says that, for any Stein filling $Y$, infinitely many of $X_{i}^{(m)}$'s remain pairwise exotic after the boundary connected sum with $Y$. 
\section{Partial twists and logarithmic transformations}\label{sec:partial}
In this section, we realize logarithmic transformations as simple monodromy substitutions. During the preparation of this paper, it turned out that our substitutions are special cases of Auroux's partial twists, though our substitutions (relations) have not been found. Therefore, we also review Auroux's partial twists and point out differences from ours. We refer to \cite{FaM} for basic relations in mapping class groups. 

\subsection{Monodromy substitution and surgery}\label{monodromy:subsection:surgery}We here recall basics of monodromy substitutions. Let $Z$ be a positive Lefschetz fibration over $D^2$ (or $S^2$) with fiber $F$ whose monodromy factorization is 
\begin{equation*}
(C_1,C_2,\dots,C_n), 
\end{equation*}
where $F$ is a compact connected oriented (possibly closed) surface, and $C_1, C_2,\dots, C_n$ are simple closed curves in $F$. Suppose that curves $D_1, D_2,\dots, D_k$ in $F$ satisfies the relation 
\begin{equation*}
(C_1, C_2,\dots,C_l)=(D_1, D_2,\dots, D_k)
\end{equation*}
in $\textnormal{Aut}(F,\partial F)$ for some $1\leq l\leq n$. Let $Z'$ be the positive Lefschetz fibration over $D^2$ with fiber $F$ whose monodromy factorization is 
\begin{equation*}
(D_1, D_2, \dots, D_k, C_{l+1},C_{l+2},\dots,C_{n}). 
\end{equation*}
We say that $Z'$ is obtained  by applying a monodromy substitution to $Z$ (cf.\ \cite{EMVHM}). 

Assume further that $C_1, C_2,\dots,C_l$ and $D_1,D_2,\dots,D_k$ are contained in the same subsurface $F'$ of $F$, and that the relation $(C_1,\dots,C_l)=(D_1, \dots, D_k)$ also holds in $\textnormal{Aut}(F',\partial F')$. Let $L$ and $L'$ denote the positive Lefschetz fibrations over $D^2$ with fiber $F'$ whose monodromy factorizations are 
\begin{equation*}
(C_1, C_2, \dots, C_l) \quad \text{and} \quad (D_1, D_2, \dots, D_k), 
\end{equation*}
respectively. Then $Z'$ is obtained from $Z$ by removing the submanifold $L$ and gluing $L'$ via the obvious identification of the induced open books on the boundary of $L$ and $L'$. The above monodromy substitution hence corresponds to a surgery of the 4-manifold $Z$. For examples of such correspondences, see \cite{EnG} and \cite{EMVHM}. 
\subsection{A PALF structure on $T^2\times D^2$}
In the rest of this section, we study logarithmic transformations from the viewpoint of PALF's. Here we 
 construct a PALF structure on $T^2\times D^2$. 

Van Horn-Morris~\cite{VHM} and subsequently Etg\"{u}~\cite{Etg} constructed a genus one open book of the Stein fillable contact structure on the 3-torus $T^3$. Note that $T^3$ has a unique Stein fillable contact structure up to contactomorphism (\cite{E3}). 
\begin{definition}
Let $\Phi$ be the element of $\textnormal{Aut}(\mathbb{S},\partial \mathbb{S})$ defined by 
\begin{equation*}
\Phi=t_{\alpha_3}^{-3}\circ t_{\alpha_2}^{-3}\circ t_{\alpha_1}^{-3}\circ t_{\delta_3}\circ t_{\delta_2}\circ t_{\delta_1}.
\end{equation*}
\end{definition}
\begin{theorem}[Van Horn-Morris~\cite{VHM} and Etg\"{u}~\cite{Etg}]
The open book $(\mathbb{S},\Phi)$ is compatible with the unique Stein fillable contact structure on $T^3$. 
\end{theorem}

Similarly to the proofs of this theorem in \cite{VHM} and \cite{Etg}, we factorize $\Phi$ into a product of right handed Dehn twists. We moreover give infinitely many factorizations. 

\begin{lemma}\label{lem: monodromy of T}$(1)$ The relation below holds in $\textnormal{Aut}(\mathbb{S},\partial \mathbb{S})$. 
\begin{equation*}
\Phi=(\gamma_{1},\, {\beta},\, {\gamma_{-1}}).
\end{equation*}

$(2)$ For any integers $a_1,a_2,a_3,d_1,d_2,d_3$, the relation 
\begin{equation*}
\Phi=(W(\gamma_{1}),\, W({\beta}),\, W({\gamma_{-1}}))
\end{equation*}
holds in $\textnormal{Aut}(\mathbb{S},\partial \mathbb{S})$, where $W$ is the element of $\textnormal{Aut}(\mathbb{S},\partial \mathbb{S})$ defined by 
\begin{equation*}
W=t_{\alpha_3}^{a_3}\circ t_{\alpha_2}^{a_2}\circ t_{\alpha_1}^{a_1}\circ t_{\delta_3}^{d_3}\circ t_{\delta_2}^{d_2}\circ t_{\delta_1}^{d_1}. 
\end{equation*}
\end{lemma}

\begin{proof}

$(1)$ Let us recall the following relation in $\text{Aut}(\mathbb{S},\partial \mathbb{S})$, which is called the star relation in \cite{Ger}. 
\begin{equation*}
t_{\delta_3}\circ t_{\delta_2}\circ t_{\delta_1}=(t_{\alpha_3}\circ t_{\alpha_2}\circ t_{\alpha_1}\circ t_{\beta})^3.
\end{equation*}
For a simple closed curve $C$ in a compact oriented surface and an element $\psi$ of the mapping class group of the surface, it is well-known that the relation $\psi \circ t_{C}\circ \psi^{-1}=t_{\psi(C)}$ holds. Note also the relation $t_{C_1}\circ t_{C_2}=t_{C_2}\circ t_{C_1}$ for simple closed curves $C_1$ and $C_2$ which are disjoint to each other. We put $V=t_{\alpha_3}\circ t_{\alpha_2}\circ t_{\alpha_1}$. Using these relations, we obtain the following relations in $\text{Aut}(\mathbb{S},\partial \mathbb{S})$, which give the claim $(1)$. 
\begin{align*}
\Phi &= V^{-2}\circ t_{\delta_3}\circ t_{\delta_2}\circ t_{\delta_1}\circ V^{-1}\\
 &= V^{-2}\circ (t_{\alpha_3}\circ t_{\alpha_2}\circ t_{\alpha_1}\circ t_{\beta})^3\circ V^{-1}\\
 &= (V^{-1}\circ t_\beta\circ V)\circ t_{\beta}\circ (V\circ t_\beta \circ V^{-1})\\
 &=t_{\gamma_{-1}}\circ t_{\beta} \circ t_{\gamma_{1}}.
\end{align*}

$(2)$ The definitions of $W$ and $\Phi$ imply $\Phi\circ W=W\circ \Phi$. Thus we obtain the following relations in $\text{Aut}(\mathbb{S},\partial \mathbb{S})$, which give the claim $(2)$. 
\begin{align*}
\Phi &= W\circ \Phi\circ W^{-1}\\
&=(W\circ t_{\gamma_{-1}}\circ W^{-1})\circ (W\circ t_{\beta} \circ W^{-1})\circ (W\circ t_{\gamma_{1}}\circ W^{-1})\\
&=t_{W(\gamma_{-1})}\circ t_{W(\beta)} \circ t_{W(\gamma_{1})}. 
\end{align*}
\end{proof}

The above factorizations give a PALF structure on $T^2\times D^2$. 
\begin{definition}Let $\mathbb{T}$ be the PALF with fiber $\mathbb{S}$ whose monodromy factorization is $(\gamma_{1},\, {\beta},\, \gamma_{-1})$.
\end{definition}
Since $T^3$ has a unique Stein filling $T^2\times D^2$ up to diffeomorphism (\cite{We}), $\mathbb{T}$ is diffeomorphic to $T^2\times D^2$. 
Here we directly construct a diffeomorphism using handlebody pictures. Moreover, we find curves in pages of the induced open book on $\partial \mathbb{T}$ which corresponds to the standard basis of $H_1(T^2\times \partial D^2;\mathbb{Z})$.

\begin{proposition}\label{prop:identification of boundary T} Regard the oriented curves $\gamma_{-1},\alpha_1,\alpha_2, \alpha_3$ in $\mathbb{S}$ as oriented knots in pairwise distinct pages of the induced open book $(\mathbb{S},\Phi)$ on $\partial \mathbb{T}$, where the order of the pages are $\gamma_{-1}<\alpha_1<\alpha_2<\alpha_3$. Then the following hold.

$(1)$ There exists an orientation preserving diffeomorphism $\mathbb{T}\to T^2\times D^2(=S^1\times S^1\times D^2)$ which maps the oriented knots $\gamma_{-1}, \alpha_1,\alpha_2$ in $\partial \mathbb{T}$ to the knots $\gamma_{-1}, \alpha_1,\alpha_2$ in $\partial(T^2\times D^2)$ shown in Figure~\ref{fig:knots_in_T}, respectively.  Consequently, the induced isomorphism maps the basis $[\alpha_1],[\alpha_2],[\gamma_{-1}]$ of $H_1(\partial \mathbb{T}; \mathbb{Z})$ to the standard basis of $H_1(S^1\times S^1\times \partial D^2; \mathbb{Z})$ as follows. 
\begin{align*}
[\alpha_1]&\mapsto [S^1\times\{pt.\}\times\{pt.\}], \quad [\alpha_2]\mapsto [\{pt.\}\times S^1\times\{pt.\}],
\\
[\gamma_{-1}]&\mapsto [\{pt.\}\times\{pt.\}\times \partial D^2].
\end{align*}

$(2)$ The above diffeomorphism $\mathbb{T}\to T^2\times D^2$ sends the surface framings of $\gamma_{-1}$, $\alpha_1$, $\alpha_2$ in $\partial \mathbb{T}$ to the Seifert framings $-1, 0, 0$ in $\partial(T^2\times D^2)$ defined by the diagram in Figure~\ref{fig:knots_in_T}, respectively. 

$(3)$ $[\alpha_1]+[\alpha_2]+[\alpha_3]=0$ in $H_1(\partial \mathbb{T};\mathbb{Z})$.
\end{proposition}
\begin{figure}[h!]
\begin{center}
\includegraphics[width=1.4in]{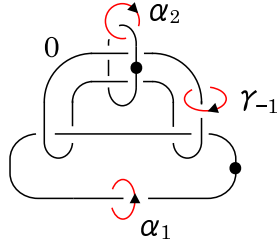}
\caption{The knots $\gamma_{-1}, \alpha_1,\alpha_2$ in $\partial(T^2\times D^2)$}
\label{fig:knots_in_T}
\end{center}
\end{figure}

\begin{proof}Figure~\ref{fig:PALF_T_0} describes the handlebody diagram of $\mathbb{T}$ induced from the PALF structure. Here $\widehat{\gamma}_{1}, \widehat{\beta}, \widehat{\gamma}_{-1}$ denote the attaching circles of the 2-handles corresponding to the vanishing cycles $\gamma_{1}, {\beta}, {\gamma}_{-1}$, respectively, and the framings of these 2-handles are $-1$ with respect to the surface framings. The oriented knots $\gamma_{-1},\alpha_1,\alpha_2$ in the figure are the oriented curves in the distinct pages of the open book on $\partial \mathbb{T}$. 

(1) and (2). Sliding the knot $\gamma_{-1}$ over the 2-handle $\widehat{\gamma}_{-1}$, we get the first picture of Figure~\ref{fig:slide_PALF_T}. We slide $\widehat{\gamma}_{-1}$ over $\widehat{\beta}$, and then slide $\widehat{\gamma}_{1}$ over $\widehat{\beta}$. The resulting diagram is the second picture of the figure, where we omit the framings. We cancel the horizontal 1-handle with the 2-handle $\widehat{\beta}$. Furthermore, we cancel the rightmost vertical 1-handle with the 2-handle $\widehat{\gamma}_{1}-\widehat{\beta}$ after sliding the 2-handle $\widehat{\gamma}_{-1}-\widehat{\beta}$ over $\widehat{\gamma}_{1}-\widehat{\beta}$ as indicated in the figure. It is easy to check that the resulting diagram coincides with the one in Figure~\ref{fig:knots_in_T} after isotopy. Thus we obtain the claim (1). One can check the claim (2) by keeping track of the surface framings. 

(3) We regard the oriented curve $\gamma_{1}$ in $\mathbb{S}$ as a knot in a page of the induced open book on $\partial \mathbb{T}$. Let us consider the diagram of $\mathbb{T}$ in Figure~\ref{fig:PALF_T_0}. Due to the open book decomposition of $\partial\mathbb{T}$, by using isotopy in $\partial \mathbb{T}$, we can change $\gamma_1$ into the surface framing of $\widehat{\gamma_1}$. Therefore, similarly to (1), we easily see that $\gamma_{1}$ is isotopic to the meridian of the attaching circle of the 2-handle $\widehat{\gamma}_1$. Then, by applying the same slidings and cancellations of handles as those in (1), we can easily check that the oriented knot $\gamma_{1}$ is isotopic to the oriented knot $\gamma_{-1}$ in Figure~\ref{fig:knots_in_T}. 
Here we note the following relations in $H_1(\mathbb{S};\mathbb{Z})\subset H_1(\partial\mathbb{T};\mathbb{Z})$.
\begin{equation*}
[\gamma_{-1}]=[\beta]-[\alpha_1]-[\alpha_2]-[\alpha_3], \quad [\gamma_{1}]=[\beta]+[\alpha_1]+[\alpha_2]+[\alpha_3]. 
\end{equation*}
Since $[\gamma_{-1}]=[\gamma_{1}]$ in $H_1(\partial\mathbb{T};\mathbb{Z})$, and $H_1(\partial\mathbb{T};\mathbb{Z})$ has no torsion, the above relations imply the claim (3). 
\end{proof}
\begin{remark}As shown in the above proof, the oriented knot $\gamma_1$ is isotopic to the oriented knot $\gamma_{-1}$ in $\partial\mathbb{T}$. 
\end{remark}
\begin{figure}[h!]
\begin{center}
\includegraphics[width=4.5in]{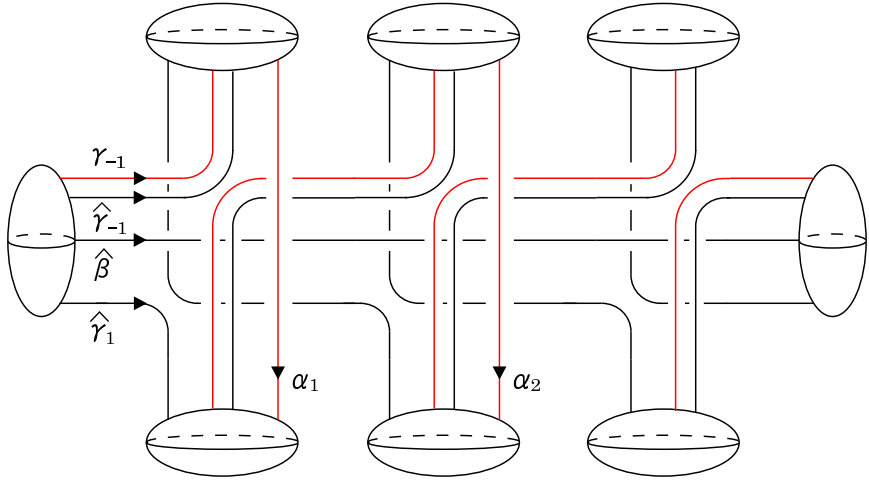}
\caption{The handlebody diagram of $\mathbb{T}$ and knots $\gamma_{-1}, \alpha_1,\alpha_2$ in $\partial\mathbb{T}$, where all framings are $-1$ with respect to the surface framings.}
\label{fig:PALF_T_0}
\end{center}
\end{figure}
\begin{figure}[h!]
\begin{center}
\includegraphics[width=4.4in]{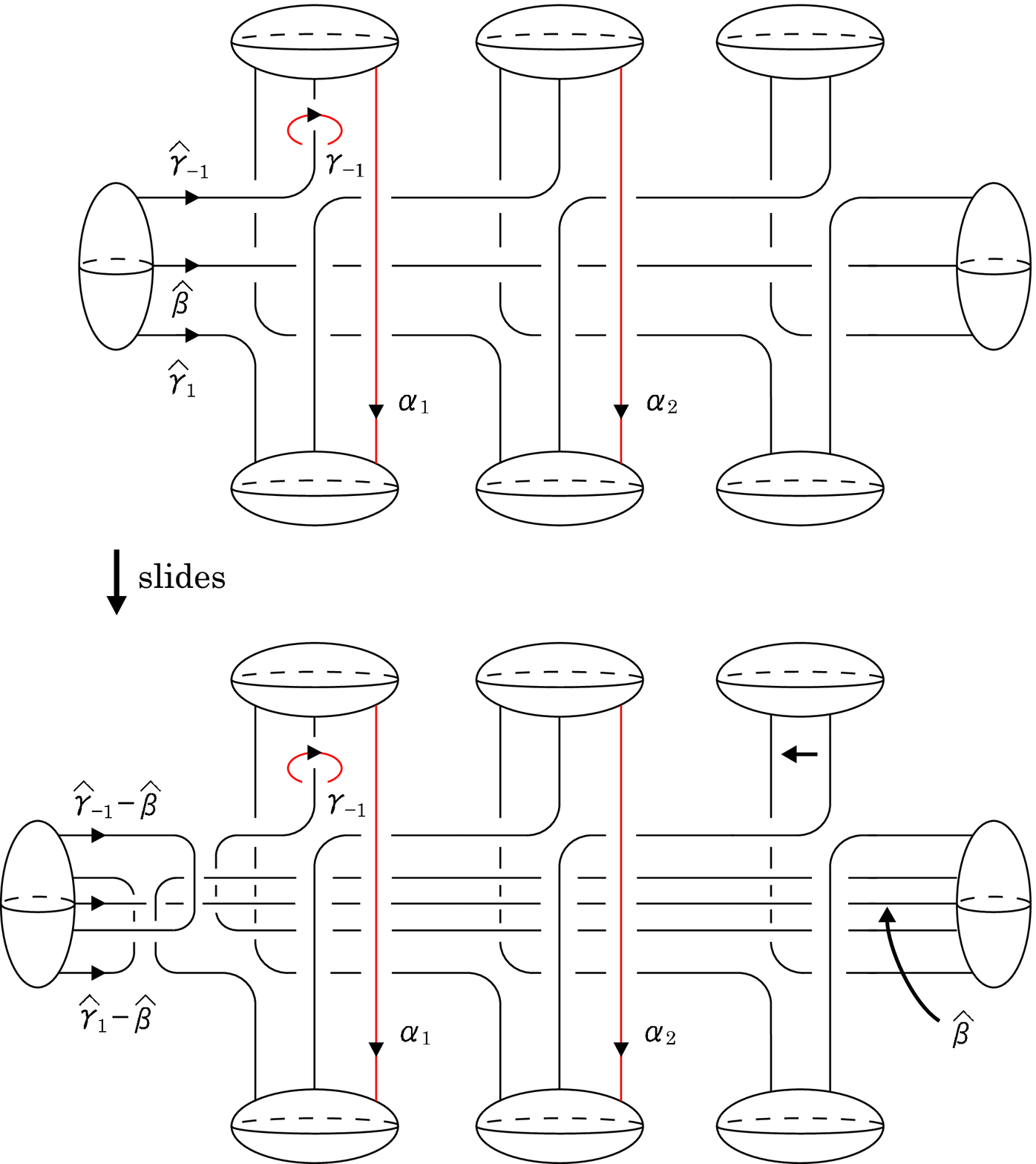}
\caption{Handle slides}
\label{fig:slide_PALF_T}
\end{center}
\end{figure}

\subsection{Monodromy substitutions and logarithmic transformations}\label{subsec:monodromy and log}We are ready to realize logarithmic transformations as monodromy substitutions. 

Let $F$ be a compact connected oriented (possibly closed) surface which contains the oriented surface $\mathbb{S}$ as a submanifold, and let $C_1,C_2,\dots,C_n$ be simple closed curves in $F$. 
We fix integers $a_1,a_2,a_3$ and define the element $W$ of $\textnormal{Aut}(F,\partial F)$ by 
\begin{equation*}
W=t_{\alpha_3}^{a_3}\circ t_{\alpha_2}^{a_2}\circ t_{\alpha_1}^{a_1}.
\end{equation*}
Let $Z$ and $Z_W$ be the positive Lefschetz fibrations over $D^2$ (or $S^2$) with fiber $F$ whose monodromy factorizations are 
\begin{equation*}
(\gamma_{1},\, {\beta},\, \gamma_{-1}, C_1,C_2,\dots,C_n) \quad \text{and} \quad (W(\gamma_{1}),\, W({\beta}),\, W({\gamma_{-1}}), C_1,C_2,\dots,C_n),
\end{equation*}
 respectively.  Lemma~\ref{lem: monodromy of T} shows that $Z_W$ is obtained by applying a monodromy substitution to $Z$. 

We prove that this monodromy substitution corresponds to a logarithmic transformation (i.e., removing a submanifold $T^2\times D^2$ and regluing $T^2\times D^2$ via a self-diffeomorphism on the boundary). 
Moreover, we describe the gluing map explicitly.  Note that $[\alpha_1], [\alpha_2], [\gamma_{-1}]$ is a basis of $H_1(\partial \mathbb{T};\mathbb{Z})$ according to Proposition~\ref{prop:identification of boundary T}, and that a self-diffeomorphism of the 3-torus $T^3$ is determined up to isotopy by its induced automorphism on $H_1(T^3;\mathbb{Z})$.

\begin{theorem}\label{thm:log:map}$Z_W$ is obtained from $Z$ by removing the submanifold $\mathbb{T}\, (\cong T^2\times D^2)$ and regluing it via the self-diffeomorphism $\varphi_W:\partial \mathbb{T}\, (\subset X) \to \partial \mathbb{T}\, (\subset X_W)$. Here $\varphi_W$ is an orientation-preserving diffeomorphism which induces the isomorphism $H_1(\partial \mathbb{T};\mathbb{Z})\to H_1(\partial \mathbb{T};\mathbb{Z})$ given by\begin{align*}
[\alpha_1] \mapsto [\alpha_1], \quad [\alpha_2]\mapsto [\alpha_2], \quad 
[\gamma_{-1}] \mapsto \; &[\gamma_{-1}]-a_1[\alpha_1]-a_2[\alpha_2]-a_3[\alpha_3]\\
=\; &[\gamma_{-1}]+(a_3-a_1)[\alpha_1]+(a_3-a_2)[\alpha_2].
\end{align*}
\end{theorem}

\begin{proof}Let $\mathbb{T}_W$ denote the PALF with fiber $\mathbb{S}$ whose monodromy factorization is $(W(\gamma_{1}),\, W({\beta}),\, W({\gamma_{-1}}))$. 
According to Subsection~\ref{monodromy:subsection:surgery}, $Z_W$ is obtained from $Z$ by removing the PALF $\mathbb{T}$ and gluing the PALF $\mathbb{T}_W$ via the diffeomorphism $f:\partial \mathbb{T}\to \partial \mathbb{T}_W$ obtained from the obvious identification of their induced open book $(S,\Phi)$.  This $f$ sends $[\alpha_1], [\alpha_2], [\gamma_{-1}]$ of $H_1(\partial \mathbb{T};\mathbb{Z})$ to the same elements of $H_1(\partial \mathbb{T}_W ;\mathbb{Z})$. 

To describe the gluing map $\varphi_W$, we apply the simultaneous conjugation to $\mathbb{T}_W$ using $W^{-1}$. This yields a diffeomorphism $g: \mathbb{T}_W\to \mathbb{T}$ which induces the isomorphism $H_1(\partial \mathbb{T}_W ;\mathbb{Z})\to H_1(\partial \mathbb{T} ;\mathbb{Z})$ given by
\begin{align*}
[\alpha_1] &\mapsto [W^{-1}(\alpha_{1})] =[\alpha_1], \qquad [\alpha_2]\mapsto [W^{-1}(\alpha_{2})] =[\alpha_2], \\
[\gamma_{-1}] &\mapsto [W^{-1}(\gamma_{-1})] =[\gamma_{-1}]-a_1[\alpha_1]-a_2[\alpha_2]-a_3[\alpha_3].
\end{align*}
Since $\varphi_W=g|_{\partial \mathbb{T}_W}\circ f$, the claim follows.
\end{proof}
\begin{remark}It follows from this theorem that the multiplicity (see \cite{GS}) of the above logarithmic transformation is one. 
\end{remark}
Having Auroux's terminology in mind (see the next subsection), we say that $Z_W$ is obtained by applying a \textit{partial $W$-twist} to $Z$ along the sub-PALF $\mathbb{T}$. When we do not specify $W$ and $\mathbb{T}$, we simply call it a \textit{partial twist}. 

We give a sufficient condition that partial twists do not change isomorphism types. 
\begin{lemma}\label{sec:partial:lem:condition:unchange} Let $\mu$ be one of the curves $\alpha_1,\alpha_2,\alpha_3$. Suppose $W=t_{\mu}^i$ for some integer $i$. If $\mu\in \{C_1,\dots, C_n\}$, then the positive Lefschetz fibration $Z_W$ is isomorphic to $Z$. In particular, $Z_W$ is diffeomorphic to $Z$.
\end{lemma}
\begin{proof}Though this lemma easily follows from the proof of Lemma 3.2 in \cite{Auroux0}, we give a proof for completeness. 
Applying elementary transformations to $Z_W$ and $Z$, we may assume $C_1=\mu$. Using elementary transformations, we change the monodromy factorization $(t_\mu^j(\gamma_{1}),\, t_\mu^j({\beta}),\, t_\mu^j({\gamma_{-1}}), \mu)$ $(j\in \mathbb{Z})$ as follows. 
\begin{align*}
(t_\mu^j(\gamma_{1}),\, t_\mu^j({\beta}),\, t_\mu^j({\gamma_{-1}}), \mu)=  \; &(\mu, t_\mu^{j-1}(\gamma_{1}),\, t_\mu^{j-1}({\beta}),\, t_\mu^{j-1}({\gamma_{-1}}))  \\
= \; &(t_\mu^{j-1}(\gamma_{1}),\, t_\mu^{j-1}({\beta}),\, t_\mu^{j-1}({\gamma_{-1}}), V(\mu))\\
= \; &(t_\mu^{j-1}(\gamma_{1}),\, t_\mu^{j-1}({\beta}),\, t_\mu^{j-1}({\gamma_{-1}}), \mu).
\end{align*}
The second factorization is obtained by moving $\mu$ to the left most. Pushing $\mu$ to the right most, we obtain the third factorization, where 
\begin{equation*}
V=t_\mu^{j-1}\circ t_{\gamma_{-1}}^{-1}\circ t_{\beta}^{-1}\circ t_{\gamma_{1}}^{-1}\circ t_\mu^{-(j-1)}. 
\end{equation*}
We obtain the last factorization, since 
\begin{equation*}
V(\mu)=(t_\mu^{j-1}\circ \Phi^{-1}\circ t_\mu^{-(j-1)})(\mu)=\Phi^{-1}(\mu)=\mu. 
\end{equation*}
The above relation implies that the monodromy factorization of $Z_W$ is obtained from that of $Z$ by using elementary transformations. The claim thus follows. 
\end{proof}
\begin{remark}Since the multiplicity of our logarithmic transformation is one, Lemma~2.2 in \cite{G_AGT} and Theorem~\ref{thm:log:map} give an alternative proof that the above $Z_W$ is diffeomorphic to $Z$. 
\end{remark}
\subsection{Auroux's partial twists}Here we briefly review Auroux's partial twisting operation (\cite{Auroux1}, \cite{Auroux2}). Let $\alpha, C_1,C_2,\dots,C_n$  be simple closed curves in a compact oriented (possibly closed) surface $F$. Assume that $t_{C_k}\circ t_{C_{k-1}}\circ \dots \circ t_{C_1}$ preserves the curve $\alpha$ for some integer $k$. Note that the relation 
\begin{equation*}
(t_{\alpha}(C_1), t_{\alpha}(C_2),\dots, t_{\alpha}(C_k))=(C_1,C_2,\dots, C_k)
\end{equation*}
 holds in $\textnormal{Aut}(F, \partial F)$. 
Let $Z$ and $Z_\alpha$ be the positive Lefschetz fibrations over $S^2$ (or $D^2$) with fiber $F$ whose monodromy factorizations are 
\begin{equation*}
(C_1, C_2,\dots, C_n) \quad \text{and} \quad (t_{\alpha}(C_1), t_{\alpha}(C_{2}),\dots, t_{\alpha}(C_k), C_{k+1}, \dots, C_{n}), 
\end{equation*}
respectively. 

The monodromy of $Z_\alpha$ is clearly obtained from that of $Z$ by taking a partial conjugation with $t_\alpha$. For this reason, Auroux said that $Z_\alpha$ is obtained by applying a partial twisting operation to $Z$. Our partial twists introduced in Subsection~\ref{subsec:monodromy and log} can be obtained by repeating Auroux's partial twists, since $\Phi$ preserves the curves $\alpha_1, \alpha_2, \alpha_3$. 
Auroux also proved that $Z_\alpha$ is a logarithmic transformation of $Z$ along a certain torus $T_{\alpha}$ (more strongly, a Luttinger surgery with the direction $\alpha$ in the case where $Z$ is closed) using a different argument. His torus $T_{\alpha}$ depends on $\alpha$, and it is not clear to the author whether our torus $T^2\times \{0\}\subset T^2\times D^2 \cong \mathbb{T}$ is isotopic to $T_{\alpha}$. 


We would like to emphasize the following differences of our realization of logarithmic transformations from Auroux's realization: (1) our relations for partial twists are concrete and simple; (2) our torus is independent of the direction $\alpha$; (3) we obtained the curves in the fiber which is a basis in $H_1(\partial \mathbb{T};\mathbb{Z})$. Indeed, these curves play important roles in our construction of exotic Stein fillings.

\section{$R^\pm$-modifications, Dehn twists and rotation numbers}\label{sec:rotation}
Here we study effects of $R$-modifications and Dehn twists on rotation numbers. In particular, we introduce $R^+$- and $R^-$-modifications. 
Throughout this section, let $F$ be a compact connected oriented surface with non-empty boundary. 

We first observe an effect of $R$-modifications on PALF's. Let $C_1, C_2, \cdots, C_n$ be (possibly homologically trivial) simple closed curves in the surface $F$, and let $Z$ be the positive Lefschetz fibration over $D^2$ with fiber $F$ whose monodromy factorization is 
\begin{equation*}
({C_1}, C_2, \dots, {C_n}). 
\end{equation*}

We fix an integer $i$ with $1\leq i\leq n$. Let $C_i'$ (resp.\ $F$) be the curve (resp.\ the oriented surface) obtained by applying an $R$-modification to $C_i$, and let $E_i$ denote the auxiliary curve of the $R$-modification. Let $Z'$ be the positive Lefschetz fibration over $D^2$ with fiber $F'$ whose monodromy factorization is 
\begin{equation*}
({C_1}, C_2,\dots, C_{i-1}, E_i, {C_i'}, C_{i+1}, C_{i+2}, \dots,  {C_n}). 
\end{equation*}
We consider the handlebody structures on $Z$ and $Z'$ induced from the PALF structures. We easily see the following. 

\begin{lemma}\label{sec:R:lem:diffeo:modification}
$(1)$ The handlebody $Z'$ is obtained from the handlebody $Z$ by decreasing the framing of the $2$-handle corresponding to the vanishing cycle $C_i$ by one. Consequently, the diffeomorphism type of $Z'$ does not depend on the choice of the $R$-modification applied to $C_i$. 

$(2)$ The fundamental group and the homology group of $Z'$ are isomorphic to those of $Z$.

$(3)$ $Z'$ can be smoothly embedded into $Z\#\overline{\mathbb{C}{P}^2}$. 
\end{lemma}
\begin{proof}We draw the handlebody diagram of $Z'$. Using the 2-handle corresponding to the vanishing cycle $E_i$, we cancel the 1-handle associated to the $R$-modification. Then we immediately see the claim (1). The claims (2) and (3) follow from (1). 
\end{proof}

In the rest of this section, we assume that the oriented surface $F$ is equipped with a handlebody diagram and the trivialization of its tangent bundle as explained in Subsection~\ref{subsec:Chern}. Due to the above lemma, we study special $R$-modifications to calculate the rotation number of the resulting curves easily. 

Let $C$ be an oriented simple closed curve in $F$. Let $\varepsilon$ be a simple proper arc in $F$ satisfying the following conditions. 
\begin{itemize}
 \item $\varepsilon$ is contained in the 0-handle $(\subset \mathbb{R}^2)$ of $F$. 
 \item $\varepsilon$ is parallel to either the $x$- or $y$-axis of $\mathbb{R}^2$. 
 \item $\varepsilon$ does not geometrically intersect with $C$. 
\end{itemize}
Attach a 1-handle to $F$ along the two end points of $\varepsilon$, and denote the resulting surface by $F'$. Let $E$ denote the simple closed curve in $F'$ obtained by attaching the core of the new 1-handle to $\varepsilon$. Take a band connected sum of $C$ and a parallel copy of $E$ in $F'$, and denote the resulting curve in $F'$ by $C'$. We orient $C'$ and $E$ so that $C'$ preserves orientations of $C$ and $E$. Trivialize the tangent bundle of $F'$ by extending the trivialization of that of $F$. 

\begin{definition}
We say that the above operation is an \textit{$R^\pm$-modification} to $C$, that $C'$ (resp.\ $F'$) is the oriented curve (the oriented surface) obtained by applying the $R^\pm$-modification to $C$, and that the oriented simple closed curve $E$ is the \textit{auxiliary curve} of the $R^\pm$-modification. 
Furthermore, we call the above operation an \textit{$R^{+}$-modification} $($resp.\ \textit{$R^{-}$-modification}$)$ to $C$, if the rotation number of $C'$ satisfies $r(C')=r(C)+1$ $($resp.\ $r(C')=r(C)-1$$)$. 
\end{definition}
\begin{remark}\label{sec:rotation:rem:def_pm}Any $R^\pm$-modification is a special case of $R$-modifications. It immediately follows from the definition that $C'$ and $E$ are homologically non-trivial in $F$. Applying Lemma~\ref{lem:comp:winding}, we see $r(E)=0$. Furthermore, any $R^\pm$-modification is either an $R^+$- or $R^-$-modification. This can be easily seen from the example below, since, by isotopy of $C'$, we may assume that the band of $C'$ between $C$ and $E$ is local and either vertical or horizontal. 
\end{remark}

 Here we give an example of $R^{-}$- and $R^{+}$-modifications. See Figure~\ref{ex_R-modification}. The upper picture is a rectangle of the 0-handle of $F$ in $\mathbb{R}^2$. The left and right sides of the figure describe the $R^-$- and $R^+$-modifications applied to $C$, respectively. In the lower rectangles, the red regions are the attaching regions of the new 1-handles of the resulting surface $F'$. Note that each auxiliary curve $E$ contains the core of the 1-handle as a subarc. One can easily check the differences between $r(C')$ and $r(C)$ by applying Lemma~\ref{lem:comp:winding} to the $+90^\circ$ half-line. Similarly, one can check $r(E)=0$. 

\begin{figure}[h!]
\begin{center}
\includegraphics[width=3.5in]{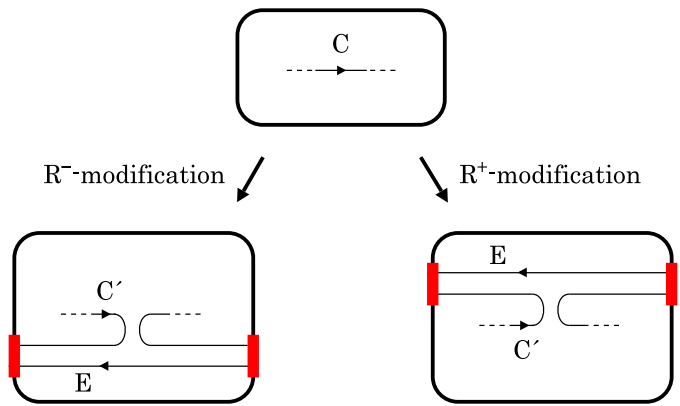}
\caption{An example of $R^-$- and $R^+$-modifications}
\label{ex_R-modification}
\end{center}
\end{figure}

We can always apply both $R^+$- and $R^{-}$-modifications to any homologically non-trivial curve in $F$, if necessary, after taking a boundary connected sum of the 0-handle of $F$ and a rectangle in $\mathbb{R}^2$. 
\begin{proposition}\label{sec:R:prop:realizable}
For any homologically non-trivial oriented simple closed curve $C$ in $F$, by taking a boundary sum of the $0$-handle in $F$ and a rectangle in $\mathbb{R}^2$ using a band in $\mathbb{R}^2$, we can enlarge the 0-handle of $F$ so that the following hold. 
\begin{itemize}
 \item The resulting  $F$ still satisfies the conditions in Subsection~\ref{subsec:Chern}. 
 \item Both $R^+$- and $R^-$-modifications can be applied to $C$ in the resulting surface $F$. Furthermore, for any simple closed curve $\alpha$ in the original $F$, the auxiliary curve $E$ of the modification does not intersect with $\alpha$. 
\end{itemize}
\end{proposition}
\begin{proof}
We begin with the lemma below. 
\begin{lemma}\label{lem:complement}Any connected component of $F-C$ contains a boundary component of $F$. 
\end{lemma}
\begin{proof}[Proof of Lemma~\ref{lem:complement}] We prove this lemma by induction on the number $n$ of boundary components of $F$. The $n=1$ case. It is well-known that a homologically non-trivial curve in a surface with connected boundary is a non-separating curve (cf. Subsection 1.3.1 in \cite{FaM}). Hence the $n=1$ case holds. Assuming the $n=k$ $(k\geq 1)$ case, we prove the $n=k+1$ case. Suppose, to the contrary, that a connected component $F_0$ of $F-C$ does not contain any boundary component of $F$. We attach one 1-handle to two distinct boundary components of $F$. Then the resulting oriented surface $F'$ has $k$ boundary components, and $C$ is still homologically non-trivial in $F'$. Due to this construction, $F_0$ is a connected component of $F'-C$, and $F_0$ contains no boundary component of $F'$. On the other hand, according to the assumption of induction, $F_0$ contains a boundary component of $F'$. This is a contradiction. Therefore, the $n=k+1$ case holds. 
\end{proof}

Let $A$ be a small subarc of $C$. By using isotopy, we may assume that $A$ is located in a rectangle of the 0-handle of $F$, that $A$ is parallel to the $x$-axis, and that the induced orientation of $A$ is the positive direction of the $x$-axis. Let $H^+$ (resp.\ $H^-$) denote the connected component of $F-C$ which is positive (resp.\ negative) normal direction for $A$. Note that $H^+$ and $H^-$ may be the same component. According to Lemma~\ref{lem:complement}, both $H^+$ and $H^-$ contain a boundary component of $F$.

Since $H^+$ contains a boundary component of $F$, there exists a small segment $\varepsilon^+$ of the boundary of a rectangle in the 0-handle of $F$ such that $\varepsilon^+$ is contained in $H^+$. Without loss of generality, we may assume that $\varepsilon^+$ is parallel to the $y$-axis as in the first picture of Figure~\ref{lem_varepsilon_v2}. Along $\varepsilon^+$, we take a boundary sum of the rectangle and a rectangle in $\mathbb{R}^2$ as in the second picture. We then attach a 1-handle to the new rectangle as in the third picture, where the red region is the attaching region of the 1-handle. Let $E$ be the unoriented simple closed curve in the third picture. Clearly $E$ does not intersect with any simple closed curve in the original $F$. 

Since $\varepsilon^+$ is contained in $H^+$, and $H^+$ is path connected, we can push a small subarc of $E$ as in the first picture of Figure~\ref{E_in_H+} by using isotopy in $H^+$. Namely, the subarc of $E$ is a push-off of $A$ to its positive normal direction. We take a band connected sum of $C$ and a push-off of $E$ as in the second picture. We denote the resulting simple closed curve by $C'$, and orient $C'$ and $E$ so that $C'$ preserves the orientations of $C$ and $E$. Since $r(E)=0$, one can check $r(C')=r(C)+1$ applying Lemma~\ref{lem:comp:winding}. Therefore this operation is an $R^+$-modification to $C$. Applying the same argument to $H^-$, we can similarly obtain an $R^-$-modification to $C$. 
\end{proof}

\begin{figure}[h!]
\begin{center}
\includegraphics[width=3.9in]{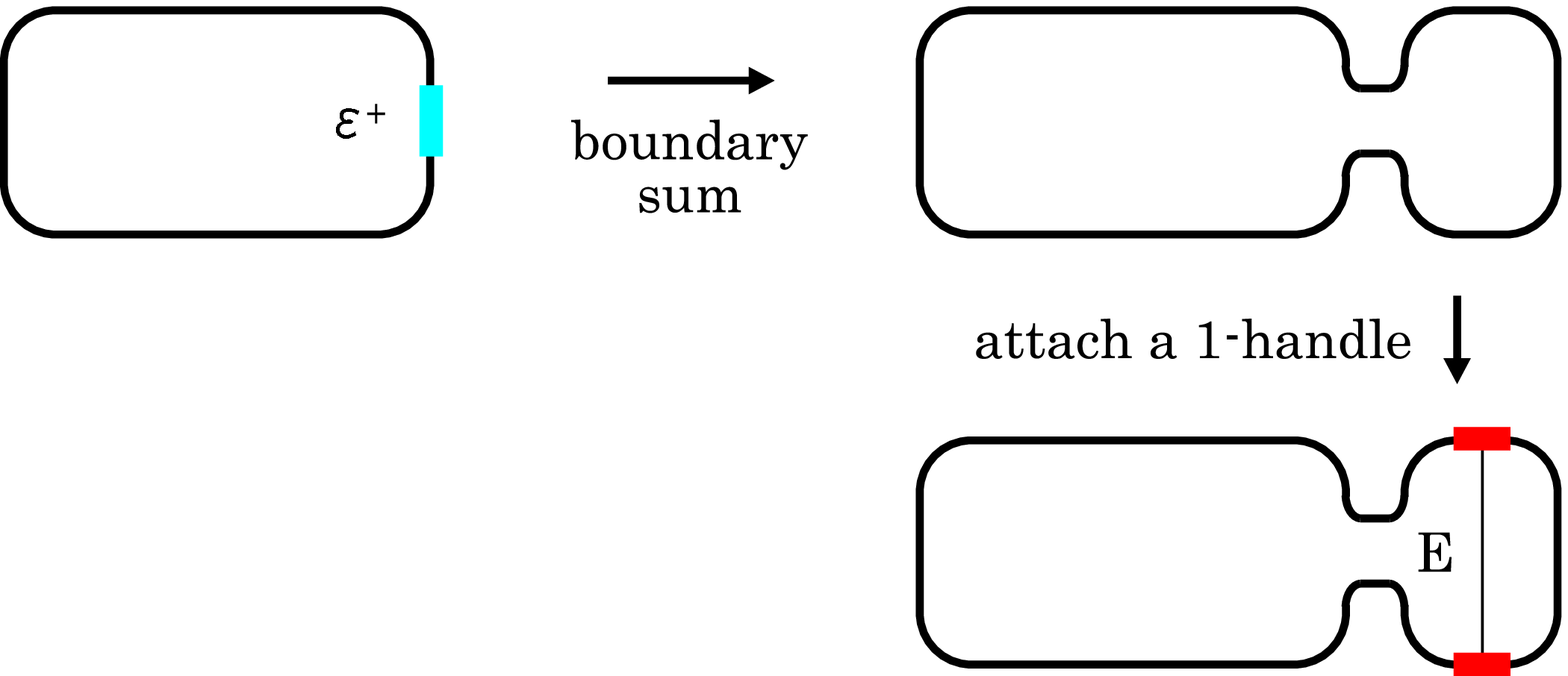}
\caption{}
\label{lem_varepsilon_v2}
\end{center}
\end{figure}

\begin{figure}[h!]
\begin{center}
\includegraphics[width=3.8in]{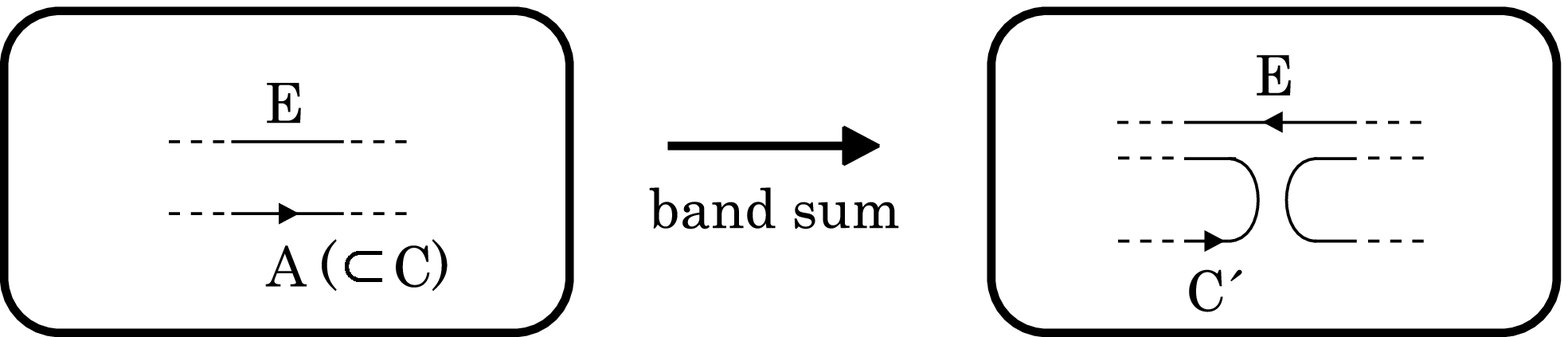}
\caption{}
\label{E_in_H+}
\end{center}
\end{figure}

Lastly we check how Dehn twists change rotation numbers. 
\begin{lemma}\label{lem:rotation:dehn}For oriented simple closed curves $C, D$ in $F$, the equalities below hold, where the orientation of $t_C^{\pm 1}(D)$ is the one induced from $D$. 
\begin{equation*}
r(t_C(D))=r(D)+Q(C,D)r(C), \quad r(t_C^{-1}(D))=r(D)-Q(C,D)r(C).
\end{equation*}
\end{lemma}
\begin{proof}
By using isotopy in a neighborhood of $C$, we may assume that all intersection points are located in a small disk as shown in the left picture of Figure~\ref{winding_Dehn}. The right picture describes a part of $t_C(D)$ in the small disk. We can now easily check the first equality by applying Lemma~\ref{lem:comp:winding} to the $+45^\circ$ half-line. The same argument shows the second equality. 
\begin{figure}[h!]
\begin{center}
\includegraphics[width=3.7in]{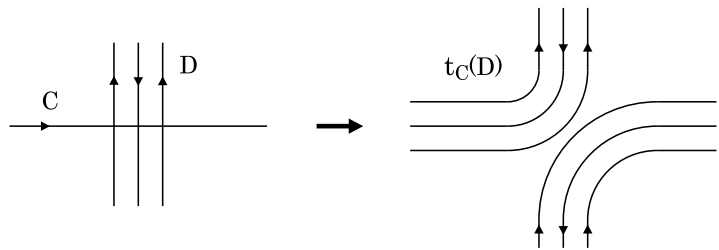}
\caption{A local picture of a Dehn twist}
\label{winding_Dehn}
\end{center}
\end{figure}
\end{proof}

\section{Proof of the main result}\label{sec:proof of main}
In this section, we prove Proposition~\ref{sec:algorithm:prop:PALF} and Theorem~\ref{sec:algorithm:thm:exotic_PALF}, from which the main result Theorem~\ref{sec:algorithm:mainthm} immediately follows. We also prove Corollary~\ref{sec:algorithm:cor:boundary sum}. We use the same symbols and the assumptions as those in Section~\ref{section:main algorithm}. In particular, $X$ denotes a fixed PALF satisfying the conditions in Subsection~\ref{subsec:S}, and $X^{(m)}$ and $X_i^{(m)}$ denote the PALF's obtained by applying Steps~\ref{step:modified PALF}--\ref{step:condition:exotic} to $X$. However, unless otherwise stated, we do not assume that an $n$-tuple $m$ of non-negative integers satisfies the conditions in Step~\ref{step:condition:exotic}. 

\subsection{Key submanifolds}\label{sec:proof:subsec:key}We first discuss topological properties of certain submanifolds of $X_i^{(m)}$'s. These submanifolds have the properties similar to those of Gompf nuclei (\cite{G0}, \cite{Ue}) and play important roles in the proof of the desired theorem. These are studied in more detail in Section~\ref{sec:ex}. 

The assumption on $X$ gives integers $j_1, j_2, j_3$ with $1\leq j_1<j_2<j_3\leq n$ such that the oriented curves $\gamma_{-1},\alpha_1,\alpha_2$ are contained in $\{C_{j_1},C_{j_2},C_{j_3}\}$.  Let $m'$ be the 3-tuple of integers defined by $m'=(m_{j_1}, m_{j_2},m_{j_3})$. 
Let $\mathbb{S}^{(m')}$ denote the surface obtained from $\mathbb{S}$ by applying the $m_{j_1}, m_{j_2}, m_{j_3}$ times $R$-modifications to $C_{j_1},C_{j_2},C_{j_3}$, respectively. We note that $\mathbb{S}^{(m')}$ is a subsurface of $\mathbb{S}^{(m)}$. 
Let $N$ be the PALF with fiber $\mathbb{S}$ whose monodromy factorization is 
\begin{equation*}
(\gamma_{1}, \beta, \gamma_{-1},C_{j_1}, C_{j_2}, C_{j_3}). 
\end{equation*}
 For an integer $i$, let $N^{(m')}_i$ be the PALF with fiber $\mathbb{S}^{(m')}$ whose monodromy factorization is
\begin{align*}
(\gamma_{1}^{(i)}, \beta^{(i)},  \gamma_{-1}^{(i)}, \, & E_1^{j_1}, E_2^{j_1}, \cdots, E_{m_1}^{j_1}, C_1(m_1), \\
 &E_1^{j_2}, E_2^{j_2}, \dots, E_{m_2}^{j_2}, C_2(m_2), E_{1}^{j_3}, E_{2}^{j_3}, \dots, E_{m_3}^{j_3}, C_{j_3}(m_3)). 
\end{align*}
The total space $N^{(m')}_i$ satisfies the following properties, similarly to Gompf nuclei. 

\begin{lemma}\label{sec:proof:lem:top:nuclei}Each $N^{(m')}_i$ $(i\in \mathbb{Z})$ satisfies the following. 

$(1)$ $\pi_1(N^{(m')}_i)\cong 1$ and $H_2(N^{(m')}_i)\cong \mathbb{Z}\oplus \mathbb{Z}$. 

$(2)$ There exists a basis $T_i, S_i$ of $H_2(N_i^{(m')};\mathbb{Z})$ satisfying the following. 
\begin{itemize}
 \item $T_i\cdot T_i=0$, \; $T_i\cdot S_i=1$, \; and \; $S_i\cdot S_i\in \{0,1\}$.
 \item $T_i$ is represented by a smoothly embedded torus in $N_i^{(m')}$. 
\end{itemize}
Furthermore, $S_i\cdot S_i=S_j\cdot S_j$, if $i\equiv j \pmod{2}$.

$(3)$ The intersection form of $N^{(m')}_i$ is unimodular. Consequently, $\partial N^{(m')}_i$ is a homology $3$-sphere. 

$(4)$ The intersection form of $N^{(m')}_i$ is isomorphic to that of $N^{(m')}_j$, if $i\equiv j\pmod{2}$. 
\end{lemma}
\begin{proof}According to Lemma~\ref{sec:R:lem:diffeo:modification}, we see that $N^{(m')}_0$ is obtained from $\mathbb{T}\cong T^2\times D^2$ by attaching 2-handles along the vanishing cycles $C_{j_1},C_{j_2},C_{j_3}$. Due to Proposition~\ref{prop:identification of boundary T}, we see that $N^{(m')}_0$ is diffeomorphic to the 4-manifold $T^2\times D^2$ with 2-handles attached along certain knots $\alpha_1', \alpha_2',\gamma_{-1}'$ in  $\partial(T^2\times D^2)$. The left picture of Figure~\ref{fig:handles_N} describes this handlebody. Here $\alpha_1', \alpha_2'$ are knots shown in the left picture, and $\gamma_{-1}'$ is a knot which itself is isotopic to the meridian of the 0-framed 2-handle of $T^2\times D^2$, though $\gamma_{-1}'$ may link with $\alpha_1', \alpha_2'$. We denote the framing coefficients of $\alpha_1',\alpha_2',\gamma_{-1}'$ by $f_1, f_2, f_{\gamma}$, respectively. 

According to Theorem~\ref{thm:log:map}, $N^{(m')}_i$ is obtained from the handlebody $N^{(m')}_0$ shown in the left picture by removing the obvious $T^2\times D^2\cong \mathbb{T}$ and regluing it via the diffeomorphism $\varphi_W:\partial \mathbb{T}\to \partial\mathbb{T}$, where $W=t_{\alpha_1}^i$. Let $\gamma_{-1}''$ be the knot in $\partial (T^2\times D^2)$ defined by $\gamma_{-1}''=\varphi_W(\gamma_{-1}')$. Since a self-diffeomorphism of $T^3$ is determined, up to isotopy, by its induced automorphism on $H_1(T^3;\mathbb{Z})$, Section~4 (especially Figures 3 and 4) of \cite{AY6} implies that the map $\varphi_W$ preserves the framed knots $\alpha_1', \alpha_2'$ in $\partial (T^2\times D^2)$, and that the framing of $\gamma_{-1}''$ is $f_{\gamma}-i$. Furthermore, $\gamma_{-1}''$ itself is isotopic to the knot shown in the right picture, though $\gamma_{-1}''$ may link with $\alpha_1',\alpha_2'$. Therefore, $N^{(m')}_i$ is obtained from $T^2\times D^2$ by attaching 2-handles along the knots $\alpha_1', \alpha_2', \gamma_{-1}''$ with the framings $f_1, f_2, f_{\gamma}-i$. The right picture of Figure~\ref{fig:handles_N} describes this handlebody. 

Let us consider the above handle decomposition of $N^{(m')}_i$. The claim (1) immediately follows from this handlebody. We check the claim (2). We slide the 2-handle $\gamma_{-1}''$ over the 2-handle $\alpha_1'$ $i$ times, so that the resulting attaching circle $\gamma_{-1}''+i\alpha_1'$ does not (algebraically) go over any 1-handle. Let $S_i',T_i$  be the classes of $H_2(N_i^{(m')};\mathbb{Z})$ represented by the 2-handle $\gamma_{-1}''+i\alpha_1'$ and the 0-framed 2-handle of $N_i^{(m')}$, respectively. Then $S_i', T_i$ is clearly a basis of $H_2(N_i^{(m')};\mathbb{Z})$, and $T_i$ is represented by a smoothly embedded torus in $N_i^{(m')}$. It is easy to check
\begin{equation*}
T_i\cdot T_i=0, \quad S_i'\cdot T_i=1, \quad S_i'\cdot S_i'=f_{\gamma}-i+i^2\cdot f_1+2i\cdot lk(\gamma_{-1}'', \alpha_1'), 
\end{equation*}
where $lk(\gamma_{-1}'', \alpha_1')$ denotes the linking number of $\gamma_{-1}''$ and $\alpha_1'$. Therefore, we can easily see that, for some integer $k_i$, the classes $S_i:=S_i'+k_i\cdot T_i$ and $T_i$ satisfy the conditions of the claim (2). The claims (3) and (4) immediately follow from (2) and (1). 
\end{proof}

\begin{figure}[h!]
\begin{center}
\includegraphics[width=3.7in]{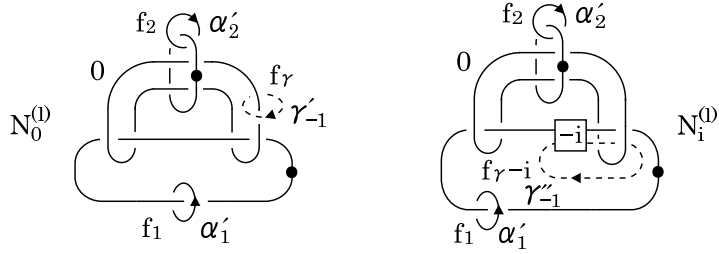}
\caption{$N^{(m')}_0$ and $N_i^{(m')}$}
\label{fig:handles_N}
\end{center}
\end{figure}

\subsection{Topological type of $X_{i}^{(m)}$}\label{sec:proof:subsec:top_type}
Using the above submanifolds, we next prove Proposition~\ref{sec:algorithm:prop:PALF}, which we restate for the reader's convenience. 
\newtheorem*{prop:X_i:again}{Proposition~\ref{sec:algorithm:prop:PALF}}
\begin{prop:X_i:again}
Fix an arbitrary $n$-tuple $m=(m_1,m_2,\dots, m_n)$ of non-negative integers. Then the following hold. 

$(1)$ The fundamental group and the homology group of each $X^{(m)}_i$ $(i\in \mathbb{Z})$ are isomorphic to those of $X$. 

$(2)$ $X_{2i}^{(m)}$'s $(i\in \mathbb{Z})$ are all homeomorphic to $X^{(m)}$. 

$(3)$ $X_{2i-1}^{(m)}$'s $(i\in \mathbb{Z})$ are pairwise homeomorphic. 

$(4)$ The open book on each $\partial X_i^{(m)}$ $(i\in \mathbb{Z})$ induced from the PALF $X_i^{(m)}$ is isomorphic to the one on $\partial X^{(m)}$ induced from the PALF $X^{(m)}$. Consequently, each $(\partial X_i^{(m)},\xi^{(m)}_i)$ is contactomorphic to $(\partial X^{(m)},\xi^{(m)})$. 

$(5)$ The intersection form of each $X^{(m)}_i$ $(i\in \mathbb{Z})$ is indefinite. Consequently, $\textnormal{sg}(\xi^{(m)}_i)\geq 1$. 

$(6)$ Each $X_{i}^{(m)}$ $(i\in \mathbb{Z})$ can be smoothly embedded into $X\#_{j=1}^n m_j\overline{\mathbb{C}{P}^2}$.

$(7)$ The PALF's $\widetilde{X}_{i}^{(m)}$'s $(i\in \mathbb{Z})$ are pairwise isomorphic. Consequently, $X_{i}^{(m)}$'s $(i\in \mathbb{Z})$ are sub-PALF's of the same PALF $\widetilde{X}_{0}^{(m)}$. 
\end{prop:X_i:again}
\begin{proof}The claims (1) and (4) follow from Lemmas~\ref{sec:R:lem:diffeo:modification} and \ref{lem: monodromy of T}, respectively. Lemma \ref{sec:proof:lem:top:nuclei} implies the former claim of (5). The latter claim of (5) follows from Etnyre's Theorem 1.2 in \cite{Et3}. 

We prove the claims (2) and (3). According to Theorem~\ref{thm:log:map}, each $X^{(m)}_i$ is a logarithmic transformation of $X^{(m)}$ along $\mathbb{T}\subset N_0^{(m')}$. Therefore, for each integers $i$ and $j$, the PALF $X^{(m)}_j$ is obtained from $X^{(m)}_i$ by removing $N_i^{(m')}$ and gluing $N_j^{(m')}$. Due to Lemma~\ref{sec:proof:lem:top:nuclei} and a theorem of Boyer (Corollary 0.9 in \cite{B}), the gluing map extends to a homeomorphism between $N_i^{(m')}$ and $N_j^{(m')}$, if $i\equiv j \pmod{2}$. This fact implies the claims (2) and (3). 

We check the claims (6) and (7). According to Lemma~\ref{sec:R:lem:diffeo:modification}, we can embed each $X_{i}^{(m)}$ into $X_i^{(0)}\#_{j=1}^n{m_j}\overline{\mathbb{C}{P}^2}$. Since Lemma~\ref{sec:partial:lem:condition:unchange} implies that $X_i^{(0)}$ is diffeomorphic to $X^{(0)}=X$, we obtain the claim (6). Lemma~\ref{sec:partial:lem:condition:unchange} also implies the claim (7). 
\end{proof}
\subsection{Relative genus function}
Here we review a simple version of the relative genus function introduced in \cite{Y5}. 

Let $Z$ be a compact connected oriented smooth 4-manifold, and let 
\begin{equation*}
g_Z:H_2(Z;\mathbb{Z})\to \mathbb{Z}
\end{equation*}
be its minimal genus function. Namely, $g_Z(\alpha)$ is the minimal genus of a smoothly embedded closed oriented connected surface in $Z$ which represents $\alpha\in H_2(Z;\mathbb{Z})$. In general, it is difficult to distinguish the genus functions of pairwise exotic 4-manifolds with $b_2\geq 2$, because of the difficulty of determining all possible identifications of their intersection forms. To avoid this issue, we use the relative genus function. 

Assume that $H_2(Z;\mathbb{Z})$ has no torsion and that $b_2(Z)\geq 2$. Put $k=b_2(Z)-1$. Let $Q$ be a $(k+1)\times (k+1)$ integral symmetric matrix, and let $\boldsymbol{g}=(g_1,g_2,\dots, g_k)$ be a $k$-tuple of non-negative integers. For an ordered basis $z=(z_0,z_1,\dots, z_{k})$ of $H_2(Z;\mathbb{Z})$, we define $G_{Z, Q,\boldsymbol{g}}(z)\in \mathbb{Z}\cup \{\infty\}$ by 
\begin{equation*}
G_{Z, Q,\boldsymbol{g}}(z)=\left\{
\begin{array}{ll}
g_Z(z_0), & \text{if the ordered basis $z$ satisfies the following. } \\
 &\: \text{$\bullet$ $Q$ is the intersection matrix of $Z$ with respect to $z$.}\\
 &\: \text{$\bullet$ $g_Z(z_i)\leq g_i$ for each $1\leq i\leq k$.}\\
\infty, & \text{otherwise.}
\end{array}
\right.
\end{equation*}
Furthermore, we define $G_Z(Q,\boldsymbol{g})\in \mathbb{Z}\cup \{\infty\}$ as the minimal number of $G_{Z, Q,\boldsymbol{g}}(z)$ for an ordered  basis $z$ of $H_2(Z;\mathbb{Z})$. Namely, 
\begin{equation*}
G_Z(Q,\boldsymbol{g})=\min\{G_{Z, Q,\boldsymbol{g}}(z)\mid \text{$z=(z_0,\dots,z_k)$ is an ordered basis of $H_2(Z;\mathbb{Z})$.} \}.
\end{equation*}
We thus obtain the function 
\begin{equation*}
G_Z:\textnormal{Sym}_{k+1}(\mathbb{Z})\times ({\mathbb{Z}}_{\geq 0})^k\to \mathbb{Z}\cup\{\infty\}, 
\end{equation*}
where $\textnormal{Sym}_{k+1}(\mathbb{Z})$ and ${\mathbb{Z}}_{\geq 0}$ denote the set of $(k+1)\times (k+1)$ integral symmetric matrices and  the set of non-negative integers, respectively. We call $G_Z$ the relative genus function of $Z$. Clearly, the relative genus function is an invariant of smooth 4-manifolds. That is, if a 4-manifold $Z'$ is diffeomorphic to $Z$ preserving the orientations, then $G_Z=G_{Z'}$. 

\subsection{Diffeomorphism type of $X_i^{(m)}$}
We distinguish diffeomorphism types of $X_i^{(m)}$'s, evaluating their relative genus functions by using the adjunction inequality. 

We first find a convenient basis of $H_2(X_i^{(m)};\mathbb{Z})$ to evaluate the relative genus function. Let $Y$ be the 4-manifold defined by 
\begin{equation*}
Y=X_i^{(m)}-\text{int}\, N_i^{(m')}. 
\end{equation*}
The diffeomorphism type of $Y$ is independent of the index $i$, since each $X_i^{(m)}$ is obtained from $X_0^{(m)}$ by removing $N_0^{(m')}$ and gluing $N_i^{(m')}$ as shown in Subsection~\ref{sec:proof:subsec:top_type}. Furthermore, since $\partial N_i^{(m')}$ is a homology 3-sphere, each $H_2(X_i^{(m)};\mathbb{Z})$ has the following natural decomposition:
\begin{equation*}
H_2(X_i^{(m)};\mathbb{Z})=H_2(N_i^{(m')};\mathbb{Z})\oplus H_2(Y;\mathbb{Z}). 
\end{equation*}
The handlebody $X_i^{(m)}$ has no 3-handles, and therefore $H_2(X_i^{(m)};\mathbb{Z})$ has no torsion. 
Let us fix a basis $v_1,v_2,\dots,v_p$ $(p\geq 0)$ of $H_2(Y;\mathbb{Z})$. Beware that $p$ can be zero. 
Using this basis, we obtain the desired ordered basis $v^{(i)}$ of $H_2(X_i^{(m)};\mathbb{Z})$ defined by
\begin{equation*}
v^{(i)}=(S_i,T_i,v_1,v_2,\dots,v_p),  
\end{equation*} 
where we regard the elements $S_i,T_i$ of $H_2(N_i^{(m')};\mathbb{Z})$ as elements of $H_2(X_i^{(m)};\mathbb{Z})$ via the above decomposition. 

Let ${Q}_i$ be the intersection matrix of $X_i^{(m)}$ with respect to the above basis $v^{(i)}$, and let $\boldsymbol{g}=(g_1,g_2,\dots,g_{p+1})$ be the $(p+1)$-tuple of non-negative integers defined by 
\begin{equation*}
(g_1,g_2,\dots,g_{p+1})=(1, g_Y(v_1), g_Y(v_2), \dots, g_Y(v_p)). 
\end{equation*}
Note that $\boldsymbol{g}$ is independent of the index $i$, and that $Q_i=Q_j$ if $i\equiv j \pmod{2}$. 

We next evaluate the value $G_{X_i^{(m)}}(Q_i,\boldsymbol{g})$ of the relative genus function of $X_i^{(m)}$. 
To achieve this, we equip a convenient PALF structure on each 4-manifold $X_i^{(m)}$ as follows. 
We regard the regular fiber of the PALF $X$ as the handlebody $\widehat{\Sigma}$ obtained in Step~\ref{step:handle}. According to Lemma~\ref{sec:R:lem:diffeo:modification} and Proposition~\ref{sec:R:prop:realizable}, we may assume that the $R$-modifications applied to each $C_j$ are $R^{\pm}$-modifications, without changing the diffeomorphism type of $X_i^{(m)}$. According to Proposition~\ref{sec:R:prop:realizable}, each auxiliary curve $E_k^{j}$ does not intersect with $\alpha_1$. Let $\widehat{\Sigma}^{(m)}$ denote the surface obtained from $\widehat{\Sigma}$ by applying these $m_1,m_2,\dots,m_n$ times $R^{\pm}$-modifications to $C_1,C_2,\dots,C_n$, respectively. 

Here let $C_j^{(i)}$ and $C_j^{(i)}(m_j)$ $(1\leq j\leq n)$ be the oriented simple closed curves in $\widehat{\Sigma}$ and $\widehat{\Sigma}^{(m)}$ defined by 
\begin{equation*}
C_j^{(i)}=t_{\alpha_1}^{-i}(C_j) \quad \text{and} \quad C_j^{(i)}(m_j)=t_{\alpha_1}^{-i}(C_j(m_j)), 
\end{equation*}
respectively, where their orientations are those induced from $C_j$ and $C_j(m_j)$. Since each $E_k^{j}$ does not intersect with $\alpha_1$, each $C_j^{(i)}(m_j)$ is obtained by applying $m_j$ times $R^{\pm}$-modifications to $C_j^{(i)}$, and each $E_k^{j}$ is the auxiliary curve of the $k$-th $R^{\pm}$-modification to $C_j^{(i)}$. Therefore, taking a simultaneous conjugation of the monodromy factorization of $X_i^{(m)}$ with $t_{\alpha_1}^{-i}$, we obtain the desired PALF structure with fiber $\widehat{\Sigma}^{(m)}$ whose monodromy factorization is
\begin{align*}
(\gamma_{1}, \beta, \gamma_{-1}, \, &E_1^{1}, E_2^{1}, \cdots, E_{m_1}^{1}, C_1^{(i)}(m_1), \\
 &E_1^{2}, E_2^{2}, \dots, E_{m_2}^{2}, C_2^{(i)}(m_2), \dots, E_{1}^{n}, E_{2}^{n}, \dots, E_{m_{n}}^{n}, C_n^{(i)}(m_n)). 
\end{align*}

Applying this PALF structure on $X_i^{(m)}$ and the adjunction inequality, we evaluate the value $G_{X_i^{(m)}}(Q_i,\boldsymbol{g})$ of the relative genus function. 

\begin{proposition}\label{sec:proof:prop:relative genus}Fix an $n$-tuple $m$ of non-negative integers satisfying the conditions in Step~\ref{step:condition:exotic}. Then, for each non-negative integer $\mu$, there exists a positive integer $I_\mu$ such that the following inequalities hold for any integer $i$ with $\lvert i\rvert \geq I_\mu$.
\begin{equation*}
\mu<G_{X_i^{(m)}}(Q_i, \boldsymbol{g})<\infty
\end{equation*}
\end{proposition}
\begin{proof}
The right inequality follows from the obvious inequalities below. 
\begin{equation*}
G_{X_i^{(m)}}(Q_i, \boldsymbol{g})\leq G_{X_i^{(m)},Q_i, \boldsymbol{g}}(v^{(i)})=g_{X_i^{(m)}}(S_i)<\infty. 
\end{equation*}

In the rest, we prove the left inequality. Here we briefly explain our procedure. In Part 1, we change the handle decomposition of each $X_i^{(m)}$ induced from the PALF structure to simplify the 2-chain group of $X_i^{(m)}$ given by the handlebody structure. In Part 2, we calculate the value of a cocycle on each 2-chain, where the cocycle is the one representing the first Chern class $c_1(X_i^{(m)})$. In Part 3, we alter the values of the cocycle on 2-chains by replacing $R^\pm$-modifications applied to $C_j$'s, i.e., by changing the Stein structure on $X_i^{(m)}$. In Part 4, we prove the desired left inequality, using the adjunction inequality and the values obtained in Part 3. 

We proceed the proof. Let us recall the condition of Step~\ref{step:condition:exotic} that oriented curves $\alpha_1,\alpha_2,\gamma_{-1}$ are contained in $\{C_1,C_2,\dots, C_n\}$. For simplicity of indices, we assume 
\begin{equation*}
C_1=\alpha_1, \quad C_2=\alpha_2 \quad \text{and} \quad C_3=\gamma_{-1}.
\end{equation*}
Though these assumptions on indices may change the diffeomorphism type of $X_i^{(m)}$, they do not affect this proof. 

\textbf{Part 1.} We first simplify the handlebody structure of $X_i^{(m)}$ induced from the PALF structure by cancelling 1-handles. 

Let us draw the handlebody diagram of $X_i^{(m)}$ induced from the PALF structure. 
We cancel all the 1-handles associated to the $R^\pm$-modifications applied to $C_j^{(i)}$'s, using 2-handles corresponding to the auxiliary curves $E_k^{j}$'s. The resulting handlebody is obtained from $\widehat{\Sigma}\times D^2$ by attaching 2-handles along the curves $\gamma_{-1}, \beta, \gamma_{1}, C_1^{(i)},C_2^{(i)},\dots, C_n^{(i)}$ in pairwise distinct pages of $\widehat{\Sigma}\times \partial D^2$. We denote the oriented attaching circles in $\widehat{\Sigma}\times \partial D^2$ of these 2-handles by $\widehat{\gamma}_{-1}, \widehat{\beta}, \widehat{\gamma}_{1}$, $\widehat{C}_1^{(i)},\widehat{C}_2^{(i)},\dots, \widehat{C}_n^{(i)}$, where the orientation of each circle is that of the corresponding curve in the page $\widehat{\Sigma}$. The framing of each of $\widehat{\gamma}_{-1}, \widehat{\beta}, \widehat{\gamma}_{1}$ is $-1$ with respect to its surface framing, and the framing of each $\widehat{C}_j^{(i)}$ is $-m_j-1$ with respect to its surface framing.  Let $H^1_{\alpha_1}, H^1_{\alpha_2}, H^1_{\alpha_3}, H^1_{\beta}$ denote the 1-handles of $\widehat{\Sigma}\times D^2\subset X_i^{(m)}$ corresponding to the 1-handles of $\widehat{\mathbb{S}}\subset \widehat{\Sigma}$ whose cocores are $\tau_{1}, \tau_{2}, \tau_{3}, \tau_{\beta}$, respectively. Note that each 2-handle $\widehat{C}_j^{(i)}$ goes over the 1-handle $H^1_{\alpha_1}$ algebraically $Q(C_j^{(i)},\tau_1)$ times. The same holds for other 2- and 1-handles. 

Here we regard the 2-handles $\widehat{\gamma}_{-1}, \widehat{\beta}, \widehat{\gamma}_{1}, \widehat{C}_1^{(i)}, \widehat{C}_2^{(i)}, \dots, \widehat{C}_n^{(i)}$ of the above handlebody as the basis of the 2-chain group of $X_i^{(m)}$ given by the handlebody. 
To simplify this group, we further change the handlebody structure of $X_i^{(m)}$, applying handle slides and cancellations similar to those in the proof of Proposition~\ref{prop:identification of boundary T} as follows. 
We cancel the 1-handle $H^1_{\beta}$ of $X_i^{(m)}$ with the 2-handle $\widehat{\beta}$. We note that, before the cancellation, we need to slide 2-handles over the 2-handle $\widehat{\beta}$ so that the resulting 2-handles do not go over the 1-handle $H^1_{\beta}$. We next cancel the 1-handle $H_{\alpha_3}^1$ with the 2-handle (corresponding to the 2-chain) $\widehat{\gamma}_1-\widehat{\beta}$, which is obtained by applying the above slide to $\widehat{\gamma}_1$. Lastly we cancel the 1-handles $H_{\alpha_1}^1$ and $H_{\alpha_2}^1$ with 2-handles $\widehat{C}_1^{(i)}$ and $\widehat{C}_2^{(i)}$, respectively. Let $\mathcal{H}_i^{(m)}$ denote the resulting handle decomposition of $X_i^{(m)}$. This handlebody yields the desired 2-chain group $\mathcal{C}_2(\mathcal{H}_i^{(m)})$. 

Due to these cancellations, we may regard the 2-chains $\widehat{\gamma}_{-1}+\widehat{\gamma}_{1}-2\widehat{\beta}, u_3, u_4, \dots, u_n$ as a basis of the 2-chain group $\mathcal{C}_2(\mathcal{H}_i^{(m)})$, where each $u_j$ $(3\leq j\leq n)$ is the 2-chain defined by
\begin{align*}
u_j&:=\widehat{C}_j^{(i)} \: - \: Q(C_j^{(i)}, \tau_\beta)\widehat{\beta} \: - \: Q(C_j^{(i)}, \tau_3)(\widehat{\gamma}_1-\widehat{\beta})\\
&\quad -\Bigl(Q(C_j^{(i)}, \tau_1)-Q(C_j^{(i)}, \tau_3)\Bigr)\widehat{C}_1^{(i)} \: - \: \Bigl(Q(C_j^{(i)},\tau_2)-Q(C_j^{(i)},\tau_3)\Bigr)\widehat{C}_2^{(i)}
\\
&=\widehat{C}_j^{(i)} \: - \: \Bigl(Q(C_j, \tau_\beta)-Q(C_j,\tau_3)\Bigr)\widehat{\beta} 
 \; - \; Q(C_j, \tau_3)\widehat{\gamma_1}\\ 
&\quad - \: \Bigl(Q(C_j,\tau_2)-Q(C_j,\tau_3)\Bigr)\widehat{C}_2^{(i)} \: - \: \Bigl(Q(C_j, \tau_1)+iQ(C_j,\alpha_1)-Q(C_j, \tau_3)\Bigr)\widehat{C}_1^{(i)}.
\end{align*}
Note that the 2-chain $u_3=\widehat{C}_3^{(i)}-2\widehat{\beta}+\widehat{\gamma}_{1}+i\widehat{C}_1^{(i)}$ is a cycle. 

\textbf{Part 2.}  Next, we calculate the value of a cocycle on each 2-chain in $\mathcal{C}_2(\mathcal{H}_i^{(m)})$, where the cocycle is the one representing the first Chern class $c_1(X_i^{(m)})$ of the Stein structure on the PALF $X_i^{(m)}$. 

According to Proposition~\ref{sec:PALF:Chern}, the class $c_1(X_i^{(m)})$  is represented by a cocycle whose value on each 2-handle of the handlebody induced from the PALF structure is the rotation number of the corresponding vanishing cycle in $\widehat{\Sigma}^{(m)}$. Let $K$ be such a cocycle. According to Remarks~\ref{sec:main:rem:step3}, \ref{sec:rotation:rem:def_pm} and Lemma~\ref{lem:rotation:dehn}, we see the following for each $j,k$. 
\begin{equation*}
r({\gamma}_{1})=0, \;\; r(\beta)=0, \;\; r({\gamma}_{-1})=0, \;\; r\bigl(C_j^{(i)}(m_j)\bigr)=r\bigl(C_j(m_j)\bigr), \;\; r(E_k^{j})=0. 
\end{equation*}
Therefore, from the handle moves in Part 1, we see that the values of $K$ on the 2-chains $\widehat{\gamma}_{-1}+\widehat{\gamma}_{1}-2\widehat{\beta}$ and $u_j$ $(3\leq j\leq n)$ obtained in Part 1 are as follows. 
\begin{align*}
 \langle K, \widehat{\gamma}_{-1} & +\widehat{\gamma}_{1}-2\widehat{\beta} \rangle = r({\gamma}_{-1})+r({\gamma}_{1})-2r(\beta)=0, \\
\langle K, u_j \rangle 
& = r\bigl(C_j^{(i)}(m_j)\bigr) \: - \: \Bigl(\sum_{k=1}^{m_j}r(E_k^{j})\Bigr) 
\: - \: \Bigl(Q(C_j, \tau_\beta)-Q(C_j,\tau_3)\Bigr)r(\beta)\\ 
&\quad  - \: Q(C_j, \tau_3)r(\gamma_1) \: - \: \Bigl(Q(C_j,\tau_2)-Q(C_j,\tau_3)\Bigr) \Bigl(r({C}_2^{(i)}(m_2))-\sum_{k=1}^{m_2}r(E_k^{2})\Bigr)\\
&\quad - \: \Bigl(Q(C_j, \tau_1)+iQ(C_j,\alpha_1)-Q(C_j, \tau_3)\Bigr)\Bigl(r({C}_1^{(i)}(m_1))-\sum_{k=1}^{m_1}r(E_k^{1})\Bigr)\\
&=r(C_j(m_j)) \: - \: \Bigl(Q(C_j,\tau_2)-Q(C_j,\tau_3)\Bigr)r({C}_2(m_2))\\
&\quad -\: \Bigl(Q(C_j, \tau_1)+iQ(C_j,\alpha_1)-Q(C_j, \tau_3)\Bigr)r({C}_1(m_1)).
\end{align*}

We now calculate the absolute value of $K$ on an arbitrary 2-chain. Let $z$ be an element of $\mathcal{C}_2(\mathcal{H}_i^{(m)})$. Then there exist unique integers $a, b_3,b_4,\dots,b_n$ such that
\begin{equation*}
z=a(\widehat{\gamma}_{-1}+\widehat{\gamma}_{1}-2\widehat{\beta})+b_3u_3+b_4u_4+\dots+b_{n}u_{n}.
\end{equation*}
Due to the above calculations, we see that the absolute value of $K$ on $z$ is as follows. 
\begin{align*}
\lvert\langle K, z \rangle\rvert &= \biggl|\sum_{j=3}^n b_j \biggl(r(C_j(m_j)) \, - \, \Bigl(Q(C_j,\tau_2)-Q(C_j,\tau_3)\Bigr)\, r({C}_2(m_2))\\
 &\qquad \qquad \quad - \, \Bigl(Q(C_j, \tau_1)+iQ(C_j,\alpha_1)-Q(C_j, \tau_3)\Bigr)\, r(C_1(m_1))\biggr)\biggr|.
\end{align*}
To simplify the right side, we define the functions $P_1, P_2, P_3:\mathcal{C}_2(\mathcal{H}_i^{(m)})\to \mathbb{Z}$ by 
\begin{align*}
P_1(z)&= -\sum_{j=3}^n b_jQ(C_j, \alpha_1),\\
P_2(z)&= \sum_{j\in J-\{1,2\}} b_j \Bigl(r(C_j(m_j)) - \Bigl(Q(C_j, \tau_1)-Q(C_j, \tau_3)\Bigr)r(C_1(m_1))\Bigr),\\
P_3(z)&= \sum_{j\in \{3,4,\dots, n\}-J} b_j r(C_j(m_j))\:- \: \sum_{j=3}^n b_j\Bigl(Q(C_j,\tau_2)-Q(C_j,\tau_3)\Bigr)r(C_2(m_2)).
\end{align*}
For the definition of the set $J$, see Step~\ref{step:condition:exotic} of the algorithm. 
These functions simplify the above equality as follows. 
\begin{equation*}
\lvert\langle K, z \rangle\rvert = \Bigl|i\cdot P_1(z)\cdot r(C_1(m_1))+P_2(z)+P_3(z)\Bigr|.
\end{equation*}

\textbf{Part 3.} Here we replace the $R^\pm$-modifications applied to $C_j$'s so that the right side of the above equality becomes a convenient form. 

Due to Lemma~\ref{sec:R:lem:diffeo:modification} and Proposition~\ref{sec:R:prop:realizable}, we can choose both $R^+$- and $R^-$-modifications as the $R^\pm$-modifications applied to $C_j$'s. Note that the handlebody structure $\mathcal{H}_i^{(m)}$ on ${X}_i^{(m)}$ does not depend on the choices of the $R^\pm$-modifications (see Lemma~\ref{sec:R:lem:diffeo:modification}), though the PALF structure on $X_i^{(m)}$ and hence $c_1(X_i^{(m)})$ depend on the choices. Consequently, the coefficients $a, b_3, b_4, \dots, b_n$ of a given 2-chain $z$ is independent of the choices. Since the different choices alter $C_j(m_j)$ and hence the values of $r(C_j(m_j))$'s, they change $K, P_2(z), P_3(z)$. 

Replacing the $R^\pm$-modifications, we alter the value of $\lvert\langle K, z \rangle\rvert$ as follows. 

\begin{lemma}\label{sec:proof:lem:equality on K}For any $z\in \mathcal{C}_2(\mathcal{H}_i^{(m)})$, we can replace the $R^\pm$-modifications applied to $C_j$'s $(1\leq j\leq n)$ so that the following hold. 
\begin{gather*}
\lvert\langle K, z \rangle\rvert = \bigl|i\cdot P_1(z)\cdot r(C_1(m_1))\bigr|+\bigl|P_2(z)\bigr|+\bigl|P_3(z)\bigr|,\\
\bigl|r(C_1(m_1))\bigr|\geq 1.
\end{gather*}
\end{lemma}
\begin{proof}[Proof of Lemma~\ref{sec:proof:lem:equality on K}]We prove the case $P_3(z)\geq 0$. Since $r(C_1)=r(\alpha_1)=0$, the conditions in Step~\ref{step:condition:exotic} of the algorithm guarantee $m_1\geq 1$. We can thus replace the $R^\pm$-modifications applied to $C_1$ so that 
\begin{equation*}
\bigl|r(C_1(m_1))\bigr|=\left\{
\begin{array}{ll}
2, &\text{if $m_1$ is even;}  \\
1, &\text{if $m_1$ is odd,}
\end{array}
\right.
\end{equation*}
and that the sign of $r(C_1(m_1))$ is the same as that of $i\cdot P_1(z)$. 
 Note that $P_1(z)$ is independent of the choices of the modifications. We thus obtain 
\begin{equation*}
i\cdot P_1(z)\cdot r(C_1(m_1))\geq 0.
\end{equation*}
 
Due to the conditions in Step~\ref{step:condition:exotic}, for each $j\in J-\{1,2\}$, we can replace the $R^\pm$-modifications applied to $C_j$ so that the sign of
\begin{equation*}
r(C_j(m_j))\: -\: \Bigl(Q(C_j, \tau_1)-Q(C_j, \tau_3)\Bigr)r(C_1(m_1)) 
\end{equation*}
is the same as that of $b_j$. Consequently, we obtain $P_2(z)\geq 0$. Since these operations do not change the value of $P_3(z)$, we obtain the desired claim. We can similarly prove the case $P_3(z)\leq 0$. 
\end{proof}
\textbf{Part 4.} Finally, we prove the desired left inequality. Let us fix an arbitrary non-negative integer $\mu$, and let $I_\mu$ be the positive integer defined by
\begin{equation*}
I_\mu = \max(\{2g_k-1-v_{k-1}\cdot v_{k-1}\mid 2\leq k\leq p+1\}\cup \{2\mu-1, 1\}).
\end{equation*}
The lemma below gives us the desired left inequality. 
\begin{lemma}\label{sec:proof_main_:lem:evaluation}For any integer $i$ with $\lvert i\rvert \geq I_\mu$, the following inequality holds. 
\begin{equation*}
\mu< G_{X_i^{(m)}}(Q_i, \boldsymbol{g}). 
\end{equation*}
\end{lemma}
\proof[Proof of Lemma~\ref{sec:proof_main_:lem:evaluation}]
Let $w=(w_0, w_1, \cdots, w_{p+1})$ be an ordered basis of $H_2(X_i^{(m)};\mathbb{Z})$. We assume the following conditions, since otherwise $G_{{X_i^{(m)}},Q_i, \boldsymbol{g}}(w)=\infty$.
\begin{itemize}
 \item $g_k\geq g_{X_i^{(m)}}(w_k)$ for each $1\leq k\leq p+1$. 
 \item $w_0\cdot w_0=S_i\cdot S_i$, \, $w_1\cdot w_1=0$, and $w_k\cdot w_k=v_{k-1}\cdot v_{k-1}$ for each $2\leq k\leq p+1$.
\end{itemize}

We first apply the adjunction inequality to $w_k$ for each $2\leq k\leq p+1$. Then, due to the conditions on the basis $w$, we obtain 
\begin{equation*}
2g_k-2\geq \lvert\langle K, w_k \rangle\rvert +v_{k-1}\cdot v_{k-1}. 
\end{equation*}
Here we regard the functions $P_1, P_2, P_3$ as the functions on $H_2(X_i^{(m)};\mathbb{Z})$, since the handlebody $\mathcal{H}_i^{(m)}$ has no 3-handles. 
The above inequality, Lemma~\ref{sec:proof:lem:equality on K} and the definition of $I_\mu$ thus give the following inequality. 
\begin{equation*}
I_\mu>\bigl|i\cdot P_1(w_k)\bigr|+\bigl|P_2(w_k)\bigr|+\bigl|P_3(w_k)\bigr|.
\end{equation*}
Since $\lvert i\rvert\geq I_\mu$, we obtain $P_1(w_k)=0$ for each $2\leq k\leq p+1$. Applying the adjunction inequality to $w_1$, we similarly obtain $P_1(w_1)=0$. 

Here recall that the 2-chain $u_3$ is a cycle. Since $w$ is a basis, there exist integers $x_0,x_1,\dots,x_{p+1}$ such that \begin{equation*}
[u_3]=x_0w_0+x_1w_1+\dots+ x_{p+1}w_{p+1}. 
\end{equation*}
The definition of $P_1$ shows $P_1([u_3])=1$, we thus get
\begin{equation*}
P_1(x_0w_0+x_1w_1+\dots+ x_{p+1}w_{p+1})=1.
\end{equation*}
Since  $P_1(w_k)=0$ for each $1\leq k\leq p+1$, it follows 
\begin{equation*}
P_1(x_0w_0)=x_0\cdot P_1(w_0)=1. 
\end{equation*}
We thus obtain $\lvert P_1(w_0)\rvert=1$. 

We next apply the adjunction inequality to $w_0$. Then, due to Lemma~\ref{sec:proof:lem:equality on K} and the fact $w_0\cdot w_0\in \{0,1\}$, we obtain the following inequality. 
\begin{equation*}
2g_{X_i^{(m)}}(w_0)-2\geq \bigl|i\cdot P_1(w_0)\bigr|+\bigl|P_2(w_0)\bigr|+\bigl|P_3(w_0)\bigr|. 
\end{equation*}
Since $\lvert i\rvert\geq I_\mu\geq 2\mu-1$ and $\lvert P_1(w_0)\rvert=1$, it follows
\begin{equation*}
2g_{X_i^{(m)}}(w_0)-2\geq 2\mu-1. \\
\end{equation*}
We thus obtain 
\begin{equation*}
G_{{X_i^{(m)}},Q_i, \boldsymbol{g}}(w)\geq g_{X_i^{(m)}}(w_0)>\mu.
\end{equation*}
Therefore the desired claim follows. This completes the proof of Proposition~\ref{sec:proof:prop:relative genus}. 
\end{proof}

Finally we prove Theorem~\ref{sec:algorithm:thm:exotic_PALF}, which we restate for convenience. 
\newtheorem*{thm:X_i:again}{Theorem~\ref{sec:algorithm:thm:exotic_PALF}}

\begin{thm:X_i:again}
Fix an $n$-tuple $m=(m_1,m_2,\dots, m_n)$ of non-negative integers satisfying the conditions in Step~\ref{step:condition:exotic}. Then the following hold. 

$(1)$ $X_{2i}^{(m)}$'s $(i\in \mathbb{Z})$ are pairwise homeomorphic Stein fillings of the same contact $3$-manifold $(\partial X^{(m)}, \xi^{(m)})$. Moreover, infinitely many of them are pairwise non-diffeomorphic.  

$(2)$ $X_{2i-1}^{(m)}$'s $(i\in \mathbb{Z})$ are pairwise homeomorphic Stein fillings of the same contact $3$-manifold $(\partial X^{(m)}, \xi^{(m)})$. Moreover, infinitely many of them are pairwise non-diffeomorphic.   

$(3)$ The fundamental group and the homology group of each $X^{(m)}_i$ $(i\in \mathbb{Z})$ are isomorphic to those of $X$. 

$(4)$ $\textnormal{sg}(\xi^{(m)})\geq 1$. Consequently, $\textnormal{sg}(\xi^{(m)})=1$, if the genus of the fiber $\Sigma^{(m)}$ of $X^{(m)}$ is one. 

$(5)$ Each $X_{i}^{(m)}$ $(i\in \mathbb{Z})$ can be smoothly embedded into $X\#_{j=1}^n m_j\overline{\mathbb{C}{P}^2}$.

$(6)$ $X_{i}^{(m)}$'s $(i\in \mathbb{Z})$ become pairwise diffeomorphic by attaching a $2$-handle to each $X_{i}^{(m)}$ along the same Legendrian knot in $(\partial X^{(m)}, \xi^{(m)})$ with the contact $-1$-framing. 
\end{thm:X_i:again}

\begin{proof}According to Proposition~\ref{sec:proof:prop:relative genus}, the relative genus function of $X_{i}^{(m)}$ is not equal to that of $X_{j}^{(m)}$, if $i\equiv j\pmod{2}$, and $\lvert i\rvert$ is sufficiently larger than $\lvert j\rvert$. Note that $Q_i=Q_j$ if $i\equiv j\pmod{2}$. Therefore, infinitely many of $X_{2i}^{(m)}$'s and $X_{2i-1}^{(m)}$'s are pairwise non-diffeomorphic. The desired claims thus immediately follow from Propositions~\ref{sec:algorithm:prop:PALF} (For (6), see also Subsection~\ref{sec:Lef:subsec:Stein}.). 
\end{proof}

Now we can easily prove Corollary~\ref{sec:algorithm:cor:boundary sum}. 
\newtheorem*{cor:boundary sum:again}{Corollary~\ref{sec:algorithm:cor:boundary sum}}

\begin{cor:boundary sum:again}
Fix an $n$-tuple $m=(m_1,m_2,\dots, m_n)$ of non-negative integers satisfying the conditions in Step~\ref{step:condition:exotic}. Then, for any Stein filling $Y$ of a contact $3$-manifold, there exists a contact structure $\zeta^{(m)}$ on $\partial X^{(m)}\#\partial Y$ satisfying the following. 

$(1)$ The boundary connected sums $X_{2i}^{(m)}\natural Y$'s $(i\in \mathbb{Z})$ are pairwise homeomorphic Stein fillings of the same contact $3$-manifold $(\partial X^{(m)}\#\partial Y, \zeta^{(m)})$. Moreover, infinitely many of them are pairwise non-diffeomorphic.  

$(2)$ $X_{2i-1}^{(m)}\natural Y$'s $(i\in \mathbb{Z})$ are pairwise homeomorphic Stein fillings of the same contact $3$-manifold $(\partial X^{(m)}\#\partial Y, \zeta^{(m)})$. Moreover, infinitely many of them are pairwise non-diffeomorphic.   

$(3)$ Each $X_{i}^{(m)}\natural Y$ $(i\in \mathbb{Z})$ can be smoothly embedded into $X\natural Y\#_{j=1}^n m_j\overline{\mathbb{C}{P}^2}$.

$(4)$ $\textnormal{sg}(\zeta^{(m)})\geq 1$. Furthermore, if $Y$ admits a PALF structure with genus zero fiber surface, and the fiber $\Sigma^{(m)}$ of $X^{(m)}$ is of genus one, then we may assume $\textnormal{sg}(\zeta^{(m)})=1$. 
\end{cor:boundary sum:again}
\begin{proof}
Since any Stein filling admits a PALF structure, there exist a bounded surface $F$ and curves $D_1,D_2,\dots,D_k$ in $F$ such that $Y$ is diffeomorphic to the PALF with fiber $F$ whose monodromy factorization is $(D_1,D_2,\dots,D_k)$. 
Let $Z$ and $\Sigma$ be the boundary connected sums $X\natural Y$ and $\mathbb{S}\natural F$, respectively. It immediately follows from a handlebody diagram of $Z$ that $Z$ admits the PALF structure with fiber $\Sigma$ whose monodromy factorization is \begin{equation*}
(\gamma_{1}, \beta, \gamma_{-1}, C_1, C_2, \dots, C_n, D_1,D_2,\dots,D_k). 
\end{equation*}
Let $l$ be the $(n+k)$-tuple of non-negative integers defined by  
\begin{equation*}
l=(m_1,m_2, \dots, m_n, 0,0,\dots, 0). 
\end{equation*}
For an integer $i$, let $Z_i^{(l)}$ denote the PALF obtained by applying Step~\ref{step:modified PALF} to $Z$. Due to its handlebody picture, we can easily see that the boundary sum $X_{i}^{(m)}\natural Y$ is diffeomorphic to $Z_i^{(l)}$. Note that the diffeomorphism type of $Z_i^{(l)}$ does not depend on the choices of $R$-modifications. Since $l$ satisfies the conditions of Step~\ref{step:condition:exotic} with respect to the PALF $Z$ (see Remark~\ref{sec:main:rem:step3}), the desired claims follow from Theorem~\ref{sec:algorithm:thm:exotic_PALF}. 
\end{proof}

\section{Examples}\label{sec:ex} Finally we construct various examples and prove theorems stated in Section~\ref{sec:intro}, demonstrating the algorithm in Section~\ref{section:main algorithm}. 

\subsection{Stein nuclei and their application}\label{sec:ex:subsec:nuclei} In this subsection, we study simple examples which we call Stein nuclei. They are variants of Gompf nuclei and share useful properties.  Applying them, we also prove Theorems~\ref{sec:intro:support} and \ref{intro:thm:most2-handlebody}. Beware that we use the symbol $N$ in Section~\ref{sec:proof of main} for a different meaning.

Let $N$ be the PALF with the fiber $\mathbb{S}$ whose monodromy factorization is 
\begin{equation*}
(\gamma_{1}, \beta, \gamma_{-1}, \gamma_{-1}, \alpha_1,\alpha_2), 
\end{equation*}
and let $m=(m_0,m_1,m_2)$ be a 3-tuple of non-negative integers. We apply Step~\ref{step:modified PALF} of the algorithm to $N$. Let $\gamma_{-1}(m_0), \alpha_1(m_1), \alpha_2(m_2)$ be the simple closed curves obtained by applying $m_0, m_1,m_2$ times $R$-modifications to $\gamma_{-1},\alpha_1,\alpha_2$, respectively. Let $E_k^{0}, E_k^{1}, E_k^{2}$ denote the auxiliary curves of $k$-th $R$-modifications applied to $\gamma_{-1}(m_0)$, $\alpha_1(m_1)$, $\alpha_2(m_2)$, respectively, and let $\mathbb{S}^{(m)}$ be the surface obtained from $\mathbb{S}$ by applying these $m_0+m_1+m_2$ times $R$-modifications. Since $\mathbb{S}$ is a surface of genus one, we may assume that the genus of $\mathbb{S}^{(m)}$ is also one. For an integer $i$, let $N^{(m)}$ and $N^{(m)}_i$ be the PALF's with fiber $\mathbb{S}^{(m)}$ whose monodromy factorizations are
\begin{align*}
(\gamma_{1}, \beta, \gamma_{-1}, \, & E_1^{0}, E_2^{0}, \cdots, E_{m_0}^{0}, \gamma_{1}(m_0), \\
 &E_1^{1}, E_2^{1}, \dots, E_{m_1}^{1}, \alpha_1(m_1), E_1^{2}, E_{2}^{2}, \dots, E_{m_{2}}^{2}, \alpha_2(m_2)),\\
(\gamma_{1}^{(i)}, \beta^{(i)}, \gamma_{-1}^{(i)}, \, & E_1^{0}, E_2^{0}, \cdots, E_{m_0}^{0}, \gamma_{1}(m_0), \\
 &E_1^{1}, E_2^{1}, \dots, E_{m_1}^{1}, \alpha_1(m_1), E_1^{2}, E_{2}^{2}, \dots, E_{m_{2}}^{2}, \alpha_2(m_2)), 
 \end{align*}
 respectively. Let $\eta^{(m)}$ and $\eta^{(m)}_i$ be the contact structures on the boundaries $\partial N^{(m)}$ and $\partial N^{(m)}_i$ induced from the Stein structures on $N^{(m)}$ and $N^{(m)}_i$, respectively. Note that $N^{(0)}=N$ and $N^{(m)}_0=N^{(m)}$. 

We study topological and smooth properties of $N_i^{(m)}$'s. Firstly, we draw handlebody diagrams of $N^{(m)}$ and $N_i^{(m)}$. According to Proposition~\ref{prop:identification of boundary T} and Lemma~\ref{sec:R:lem:diffeo:modification}, we see that $N^{(m)}$ has the handle decomposition in the left picture of Figure~\ref{fig:Stein_nuclei}. The obvious $T^2\times D^2$ in the picture is the submanifold $\mathbb{T}$ of $N^{(m)}$. According to Theorem~\ref{thm:log:map}, the PALF $N_i^{(m)}$ is a logarithmic transformation of $N^{(m)}$ along the obvious $T^2\times D^2$. Applying the procedure of a logarithmic transformation given in Section~4 of \cite{AY6}, we obtain the handlebody diagram of $N_i^{(m)}$ in the right picture. We check the boundary 3-manifold. Cancelling the 1-handles, we get the diagram of $N^{(m)}$ in the left picture of Figure~\ref{fig:boundary_Stein_nuclei}. Applying the slam-dunk operation, we have a Dehn surgery diagram of the boundary $\partial N^{(m)}$ in the right picture.

\begin{figure}[h!]
\begin{center}
\includegraphics[width=4.2in]{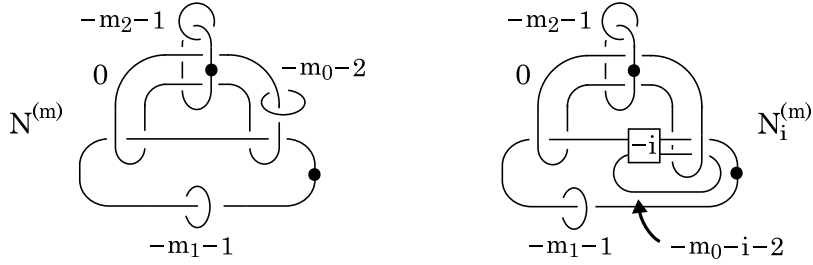}
\caption{$N^{(m)}$ and $N_i^{(m)}$}
\label{fig:Stein_nuclei}
\end{center}
\end{figure}

\begin{figure}[h!]
\begin{center}
\includegraphics[width=4.2in]{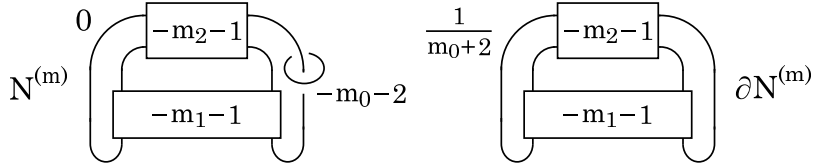}
\caption{$N^{(m)}$ and $\partial N^{(m)}$}
\label{fig:boundary_Stein_nuclei}
\end{center}
\end{figure}

Secondly, we consider the case $m_1=m_2=0$. In this case, $m$ does not satisfy the conditions of Step~\ref{step:condition:exotic}. Indeed, according to Lemma~\ref{sec:partial:lem:condition:unchange}, the PALF $N_i^{(m)}$ is isomorphic (hence diffeomorphic) to $N^{(m)}$ for any integer $i$. The diagram of $N^{(m)}$ in Figure~\ref{fig:Stein_nuclei} shows that the PALF $N^{(m)}$ is diffeomorphic to a Gompf nucleus, which was introduced independently by Gompf \cite{G0} and Ue \cite{Ue}. We note that each $N^{(m)}$ admits a Stein structure. It is well-known that Gompf nuclei are useful small building blocks for constructing various exotic smooth 4-manifolds by logarithmic transformations. However, exotic Stein fillings cannot be constructed from Gompf nuclei by any logarithmic transformation along the obvious $T^2\times D^2$. This can be seen as follows. In the case where the multiplicity of the logarithmic transformation is zero, the resulting manifold has an embedded 2-sphere with the self-intersection number $-1$ (\cite{G0}). Therefore, it does not admit any Stein structure. Since the obvious torus is contained in the cusp neighborhood, any logarithmic transformation with the multiplicity one does not change their diffeomorphism types  (\cite{G0}). Although each Gompf nucleus produces its infinitely many exotic copies by using logarithmic transformations of the multiplicities larger than one (\cite{G0}, \cite{Ue}), none of them admits any Stein structure (\cite{Y5}). 

Lastly, we consider the case $m_1\geq 1$. In this case, $m$ satisfies the conditions of Step~\ref{step:condition:exotic} (see Remark~\ref{sec:main:rem:step3}). Therefore, in contrast to Gompf nuclei, each $N^{(m)}$ produces infinitely many exotic Stein fillings by using logarithmic transformations (of the multiplicity one). For this reason, we call each 4-manifold $N^{(m)}_i$ a \textit{Stein nucleus} for $m_1\geq 1$. We hope that Stein nuclei become useful objects similarly to Gompf nuclei. 

We list basic properties of the Stein nuclei $N_i^{(m)}$'s. We first see topological properties. 

\begin{proposition}\label{subsec:nuclei:prop:nuclei:top}For an arbitrary $3$-tuple $m=(m_0,m_1,m_2)$ of non-negative integers, the following hold. 

$(1)$ $\pi_1(N^{(m)}_{i})\cong 1$ and $H_2(N^{(m)}_{i};\mathbb{Z})\cong \mathbb{Z}\oplus \mathbb{Z}$, for each integer $i$. 

$(2)$ The intersection form of each $N^{(m)}_{i}$ $(i\in \mathbb{Z})$ is unimodular and indefinite. Consequently, the boundary $\partial N^{(m)}_{i}$ is a homology $3$-sphere. 

$(3)$ The intersection form of $N^{(m)}_{i}$ is even, if and only if $m_0\equiv im_1\pmod{2}$. 

$(4)$ For each integers $i,j$, an orientation-preserving homeomorphism $\partial N_i^{(m)}\to \partial N_j^{(m)}$ extends to a homeomorphism $N_i^{(m)}\to N_j^{(m)}$, if and only if $im_1\equiv jm_1 \pmod{2}$. 

\end{proposition}
\begin{proof}Due to the diagram of $N_i^{(m)}$ in Figure~\ref{fig:Stein_nuclei}, the claim (1) immediately follows. Furthermore, we easily get a basis $T_i,S_i$ of $H_2(N_i^{(m)};\mathbb{Z})$ such that
\begin{equation*}
T_i\cdot T_i=0, \quad S_i\cdot S_i=-m_0-i-2-i^2(m_1+1), \quad S_i\cdot T_i=1. 
\end{equation*}
This fact implies the claims (2) and (3). Due to the classification theorem of intersection forms, the claims (2) and (3) imply that the intersection form of $N_i^{(m)}$ is isomorphic to that of $N_j^{(m)}$ if and only if $im_1\equiv jm_1\pmod{2}$. Since each $N_i^{(m)}$ is simply connected, and its boundary is a homology 3-sphere, this fact and Boyer's theorem (Corollary 0.9 in \cite{B}) show the claim (4). 
\end{proof}

We secondly state smooth properties. These propositions immediately follow from Proposition~\ref{sec:algorithm:prop:PALF} and Theorem~\ref{sec:algorithm:thm:exotic_PALF}. 
\begin{proposition}\label{subsec:nuclei:prop:nuclei:smooth_arbitrary} For any $3$-tuple $m=(m_0,m_1,m_2)$ of non-negative integers and any integer $i$, the following hold. 

$(1)$ $N_i^{(m)}$ can be smoothly embedded into $N\#(m_0+m_1+m_2) \overline{\mathbb{C}{P}^2}$.

$(2)$ The contact $3$-manifold $(\partial N^{(m)}_i, \eta^{(m)}_i)$ is contactomorphic to $(\partial N^{(m)}, \eta^{(m)})$. Furthermore, its support genus is one.
\end{proposition}

\begin{proposition}\label{subsec:nuclei:prop:nuclei} Fix an arbitrary $3$-tuple $m=(m_0,m_1,m_2)$ of non-negative integers with $m_1\geq 1$. Then the following hold. 

$(1)$ $N_{2i}^{(m)}$'s $(i\in \mathbb{Z})$ are pairwise homeomorphic Stein fillings of the same contact $3$-manifold $(\partial N^{(m)}, \eta^{(m)})$. Moreover, infinitely many of them are pairwise non-diffeomorphic. 

$(2)$ $N_{2i-1}^{(m)}$'s $(i\in \mathbb{Z})$ are pairwise homeomorphic Stein fillings of the same contact $3$-manifold $(\partial N^{(m)}, \eta^{(m)})$. Moreover, infinitely many of them are pairwise non-diffeomorphic.
\end{proposition}

Lastly, we check their boundaries. 
\begin{proposition}\label{subsec:nuclei:prop:boundary}For any $3$-tuple $m=(m_0,m_1,m_2)$ of non-negative integers with $m_1\geq 1$, the following hold. 

$(1)$ $\partial N^{(m)}$ is a hyperbolic $($hence irreducible$)$ $3$-manifold. 

$(2)$ For a non-negative integer $m_0'$, put $m'=(m_0', m_1, m_2)$. Then $\partial N^{(m)}$ is homeomorphic to $\partial N^{(m')}$ preserving the orientations, if and only if $m_0'=m_0$. 
\end{proposition}
\begin{proof}
(1) Figure~\ref{fig:boundary_Stein_nuclei} shows that $\partial N^{(m)}$ is obtained by Dehn surgery along a 2-bridge knot with the non-integer coefficient. The claim thus follows from the main result of \cite{BW}. 

(2) According to the right picture of Figure~\ref{fig:boundary_Stein_nuclei}, the boundaries $\partial N^{(m)}$ and $\partial N^{(m')}$ are obtained by Dehn surgeries along the same knot with two positive surgery coefficients $\frac{1}{m_0+2}$ and $\frac{1}{m_0'+2}$, respectively. Thus the claim immediately follows from Wu's result \cite{Wu} on the cosmetic surgery conjecture. 
\end{proof}

Summarizing Propositions~\ref{subsec:nuclei:prop:nuclei:top}--\ref{subsec:nuclei:prop:boundary}, we immediately see that Stein nuclei satisfy the conditions of Theorem~\ref{sec:intro:support}. 
\newtheorem*{thm:support:again}{Theorem~\ref{sec:intro:support}}

\begin{thm:support:again} There exist infinitely many pairwise non-homeomorphic contact $3$-manifolds of support genus one each of which admits infinitely many pairwise homeomorphic but non-diffeomorphic simply connected Stein fillings with $b_2=2$. Furthermore, each of these $3$-manifolds is a hyperbolic $($hence irreducible$)$ homology $3$-sphere. 
\end{thm:support:again}

\begin{remark}In \cite{AY6}, Akbulut and the author earlier studied certain Stein handlebodies which are diffeomorphic to some of $N_i^{(m)}$'s, without using PALF's. Indeed, in the case $m=(1,1,0)$, we proved that infinitely many of $N^{(m)}_i$'s are pairwise exotic Stein fillings of the same contact 3-manifold by using properties of Stein handlebodies. 
\end{remark}

Here recall that a 2-handlebody means a handlebody obtained from the 0-handle by attaching 1- and 2-handles. Due to Stein nuclei and Corollary~\ref{sec:algorithm:cor:boundary sum}, we can now easily prove Theorem~\ref{intro:thm:most2-handlebody}. 
\newtheorem*{thm:2-handlebody:again}{Theorem~\ref{intro:thm:most2-handlebody}}

\begin{thm:2-handlebody:again} Let $X$ be a compact oriented $4$-dimensional $2$-handlebody, and let $Z$ be either $X\#S^2\times S^2$ or $X\#\mathbb{C}P^2\#\overline{\mathbb{C}{P}^2}$. Then there exist infinitely many pairwise homeomorphic but non-diffeomorphic Stein fillings of the same contact 3-manifold such that the fundamental group, the homology group, the homology group of the boundary, and the intersection form of each filling are isomorphic to those of $Z$.  Furthermore, for some positive integer $k$, all of these fillings can be smoothly embedded into the same manifold $Z\#k\overline{\mathbb{C}{P}^2}$. 
\end{thm:2-handlebody:again}
\begin{proof}By isotopy, we may assume that 2-handles of $X$ are attached along a Legendrian link. Applying $W^+$-modifications in \cite{AY5} to each 2-handle, we can increase the Thurston-Bennequin number of the attaching circle of each 2-handle (Proposition 4.7 in \cite{AY5}). By adding zig-zags to the attaching circle, we may assume that the framing of each 2-handle is Thurston-Bennequin number $-1$. Therefore, the resulting handlebody $X'$ admits a Stein structure (\cite{G1}). Propositions 4.2 and 4.5 in \cite{AY5} show that the topological invariants of the Stein filling $X'$ coincide with those of $X$, and that $X'$ can be embedded into $X$. 

Let us consider Stein nuclei. Let fix a 3-tuple $m=(m_0,m_1,m_2)$ of non-negative integers satisfying 
\begin{equation*}
m_0\geq 1, \quad m_0\equiv 1 \pmod{2},\quad m_1\equiv 1 \pmod{2}.  
\end{equation*}
These conditions on $m$ and Proposition~\ref{subsec:nuclei:prop:nuclei:top} imply that the topological invariants of each $N_{i}^{(m)}$ coincide with those of $\mathbb{C}{P}^2\#\overline{\mathbb{C}{P}^2}-D^4$ (resp.\ $S^2\times S^2-D^4$), if $i$ is an even (resp.\ odd) integer. 

We next consider $N$. See the diagram of $N^{(m)}$ in Figure~\ref{fig:boundary_Stein_nuclei}, and let $N'$ denote the handlebody obtained from the diagram by replacing the framing $-m_0-2$ with $0$. Then we easily see that $N'$ is diffeomorphic to $S^2\times S^2-D^4$. Clearly, $N=N^{(0)}$ can be embedded into $N'\#2\overline{\mathbb{C}{P}^2}$. Consequently, $N$ can be embedded into $S^2\times S^2\#2\overline{\mathbb{C}{P}^2}$, which is diffeomorphic to $\mathbb{C}{P}^2\#3\overline{\mathbb{C}{P}^2}$. 
Therefore Corollary~\ref{sec:algorithm:cor:boundary sum} shows that $X'\natural N_{i}^{(m)}$'s give the desired Stein fillings. 
\end{proof}
\begin{remark}
One might expect that, if a PALF contains a Stein nucleus $N^{(m)}$ $(m_1\geq 1)$ as a sub-PALF, then it produces infinitely many exotic Stein fillings by partial twists along $\mathbb{T}\subset N^{(m)}$. However, this claim does not always hold. This can be seen as follows. Add curves $\alpha_1,\alpha_2,\alpha_3$ to the monodromy factorization of $N^{(m)}$. Then, due to Lemma~\ref{sec:partial:lem:condition:unchange}, any partial twist along $\mathbb{T}\subset N^{(m)}$ does not change the isomorphism type of the resulting PALF. Similarly, the corresponding claim for Stein handlebodies does not always hold (apply Lemma~2.2 in \cite{G_AGT}). 
\end{remark}

\subsection{Non-homeomorphic Stein fillings with small $b_2$} 
In the rest of this section, we construct more examples of infinitely many (not necessarily exotic) Stein fillings in the case where boundary contact 3-manifolds are of support genus one. 

Here we construct simple examples, which yield Theorem~\ref{sec:intro:thm:non-homeo_simple}. 
\newtheorem*{thm:non-homeo_simple:again}{Theorem~\ref{sec:intro:thm:non-homeo_simple}}

\begin{thm:non-homeo_simple:again}There exist infinitely many pairwise non-homeomorphic contact $3$-manifolds of support genus one each of which admits infinitely many pairwise non-homeomorphic Stein fillings with $b_1=0$ and $b_2=1$. Furthermore, each of these $3$-manifolds is irreducible and toroidal. 
\end{thm:non-homeo_simple:again}

Let $L$ be the PALF with fiber $\mathbb{S}$ whose monodromy factorization is 
\begin{equation*}
(\gamma_{1}, \beta, \gamma_{-1}, \gamma_{-1}, \alpha_2).
\end{equation*}
Let fix an arbitrary $2$-tuple $l=(l_1,l_2)$ of non-negative integers. For an integer $i$, let $L^{(l)}$ and $L_i^{(l)}$ be the PALF's obtained by applying Step~\ref{step:modified PALF} of the algorithm to $L$. Let $\iota^{(l)}$ and $\iota^{(l)}_i$ denote the contact structures on $\partial L^{(l)}$ and $\partial L^{(l)}_i$ induced from the Stein structures on $L^{(l)}$ and $L^{(l)}_i$, respectively. 
We note $L^{(0)}=L$. We may assume that the fiber surface $\mathbb{S}^{(l)}$ of the PALF's $L^{(l)}$ and $L_i^{(l)}$ is of genus one. 

We can easily see the following. The claim (5) was kindly pointed out by Yuichi Yamada. 

\begin{proposition}\label{sec:ex:prop:non-homeo}Fix an arbitrary $2$-tuple $l=(l_1,l_2)$ of non-negative integers. Then the following hold. 

$(1)$ Each $(\partial L_i^{(l)}, \iota^{(l)}_i)$ $(i\in \mathbb{Z})$ is contactomorphic to $(\partial L^{(l)}, \iota^{(l)})$. Furthermore, its support genus is one. In addition, each 3-manifold $\partial L_i^{(l)}$ $(i\in \mathbb{Z})$ is toroidal and irreducible. 

$(2)$ $\pi_1(L_i^{(l)})\cong \mathbb{Z}/i\mathbb{Z}$ for each integer $i$. 

$(3)$ $H_2(L_i^{(l)};\mathbb{Z})\cong \mathbb{Z}$, $($resp.\ $\mathbb{Z}\oplus \mathbb{Z}$$)$ if $i\in \mathbb{Z}-\{0\}$ $($resp.\ $i=0$$)$. 

$(4)$ The intersection form of $L_i^{(l)}$ is degenerate $($resp.\ indefinite$)$, if $i\in \mathbb{Z}-\{0\}$ $($resp.\ $i=0$$)$.   

$(5)$ For a $2$-tuple $l'=(l_1', l_2')$ of non-negative integers, if 
 the boundary $\partial L^{(l')}$ is homeomorphic to $\partial L^{(l)}$, then $(l_1'+2)(l_2'+1)=(l_1+2)(l_2+1)$. 
\end{proposition}
\begin{proof}We obtain the handlebody diagrams of $L^{(l)}$ and $L_i^{(l)}$ in Figure~\ref{fig:non_homeo}, similarly to the Stein nuclei. This diagram implies the claims (2)--(4). 

Here we draw a Dehn surgery diagram of $\partial L^{(l)}$. Let us consider the diagram of $L^{(l)}$ in the left picture of Figure~\ref{fig:non_homeo}. Exchanging the middle $0$ and the lower dot, and cancelling the 1-handles, we obtain the diagram in Figure~\ref{fig:boundary_non-homeo}. These operations do not change the boundary, and hence this picture is a Dehn surgery diagram of $\partial L^{(l)}$. Since the knot in the picture is a 2-bridge knot, Lemmas 2.1 and 2.2 in \cite{BW} show that the 3-manifold $\partial L^{(l)}$ is irreducible and toroidal. The claim (1) thus follows from Lemma~\ref{lem: monodromy of T}, the claim (4) and Etnyre's theorem in \cite{Et3}. 

Due to the maximal abelian torsion, if two knots produce the same 3-manifold by 0-surgeries, then they have the same Alexander polynomial (see Examples and Remarks 17.4 in \cite{Tra}). Since the Alexander polynomial of the knot in the right picture of Figure~\ref{fig:boundary_non-homeo} is 
\begin{equation*}
(l_1+2)(l_2+1)t - (2(l_1+2)(l_2+1)-1) + (l_1+2)(l_2+1)t^{-1},
\end{equation*}
the claim (3) follows. 
\end{proof}
\begin{figure}[h!]
\begin{center}
\includegraphics[width=3.7in]{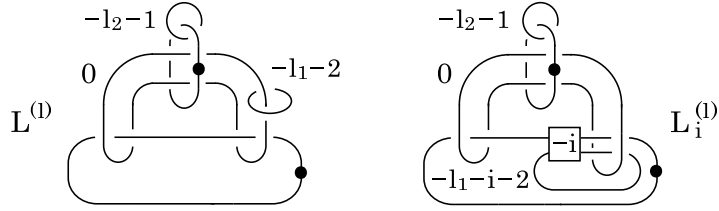}
\caption{$L^{(l)}$ and $L_i^{(l)}$}
\label{fig:non_homeo}
\end{center}
\end{figure}
\begin{figure}[h!]
\begin{center}
\includegraphics[width=1.2in]{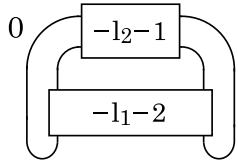}
\caption{$\partial L^{(l)}$}
\label{fig:boundary_non-homeo}
\end{center}
\end{figure}
\begin{proof}[Proof of Theorem~\ref{sec:intro:thm:non-homeo_simple}]
Due to Proposition~\ref{sec:ex:prop:non-homeo}, we immediately see that $L_i^{(l)}$'s satisfy the conditions of Theorem~\ref{sec:intro:thm:non-homeo_simple}. 
\end{proof}

\subsection{Exotic and non-homeomorphic Stein fillings}Lastly we construct examples in Theorem~\ref{sec:intro:thm:non-homeo} by combining the Stein nuclei and $L_i^{(l)}$'s. 

\newtheorem*{thm:non-homeo_exotic:again}{Theorem~\ref{sec:intro:thm:non-homeo}}
\begin{thm:non-homeo_exotic:again}There exist infinitely many pairwise non-homeomorphic contact $3$-manifolds of support genus one each of which admits infinitely many Stein fillings $Z_{i,j}$'s $(i,j\in \mathbb{N})$ satisfying the following: for each fixed  $j\in \mathbb{N}$, infinitely many Stein fillings $Z_{i,j}$'s $(i\in \mathbb{N})$ are pairwise homeomorphic but non-diffeomorphic, and for each fixed  $i\in \mathbb{N}$, infinitely many Stein fillings $Z_{i,j}$'s $(j\in \mathbb{N})$ are pairwise non-homeomorphic. 
\end{thm:non-homeo_exotic:again}

Let $\mathbb{E}$ be the compact surface of genus one with six boundary components in Figure~\ref{fig:homeo_exotic}, where each 1-handle is attached to the obvious disk either vertically or horizontally along the red regions. Let $\alpha_1,\alpha_2,\alpha_3,\beta,a_1, a_2, a_3$ be the oriented simple closed curves in $\mathbb{E}$ as shown in the picture. We orient $\mathbb{E}$ so that $Q(\alpha_1, \beta)=+1$. Let $h^1_{\alpha_1}, h^1_{\alpha_2}, h^1_{\alpha_3}, h^1_{\beta},h^1_{a_1}, h^1_{a_2}, h^1_{a_3}$ denote the obvious 1-handles of $\mathbb{E}$ whose cocores intersect $\alpha_1, \alpha_2,\alpha_3,\beta, a_1,a_2,a_3$, respectively. Let $\mathbb{S}_\alpha$ (resp.\ $\mathbb{S}_a$) be the subsurface of $\mathbb{E}$ which consists of the obvious disk and the 1-handles $h^1_{\alpha_1}, h^1_{\alpha_2}, h^1_{\alpha_3}, h^1_{\beta}$ (resp.\ $h^1_{a_1}, h^1_{a_2}, h^1_{a_3},h^1_{\beta}$). The subsurfaces $\mathbb{S}_\alpha$ and $\mathbb{S}_a$ are clearly diffeomorphic to $\mathbb{S}$ (see also the diagram of $\widehat{\mathbb{S}}$). 

For integers $i,j$, we define the simple closed curves $\gamma_{i}, \gamma_{i}^{(j)}, \beta^{(j)}, \rho_{i}, \rho_{i}^{(j)}, b^{(j)}$ in $\mathbb{E}$ by
\begin{align*}
\gamma_{i}&=(t_{\alpha_3}\circ t_{\alpha_2}\circ t_{\alpha_1})^i(\beta), \quad \gamma_{i}^{(j)}=t_{\alpha_1}^j(\gamma_{i}), \quad \beta^{(j)}=t_{\alpha_1}^j(\beta),\\
\rho_{i}&=(t_{a_3}\circ t_{a_2}\circ t_{a_1})^i(\beta), \quad \rho_{i}^{(j)}=t_{a_1}^j(\gamma_{i}), \quad b^{(j)}=t_{a_1}^j(\beta). 
\end{align*}

\begin{figure}[h!]
\begin{center}
\includegraphics[width=3.1in]{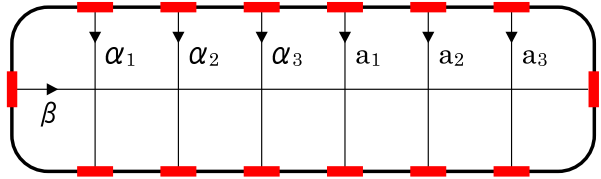}
\caption{The surface $\mathbb{E}$ of genus one with six boundary components}
\label{fig:homeo_exotic}
\end{center}
\end{figure}

For an integer $j$, let $P$ and $P_j$ be the PALF's with fiber $\mathbb{E}$ whose monodromy factorizations are
 \begin{gather*}
(\gamma_{1}, \beta, \gamma_{-1}, \gamma_{-1}, \alpha_1,\alpha_2, \rho_1, \beta, \rho_{-1}, \rho_{-1}, a_2), \\(\gamma_{1}, \beta, \gamma_{-1}, \gamma_{-1}, \alpha_1,\alpha_2, \rho_1^{(j)}, \beta^{(j)}, \rho_{-1}^{(j)}, \rho_{-1}, a_2), 
\end{gather*}
respectively. Note that $P_j$ is obtained from $P$ by a partial twist along the subsurface $\mathbb{S}_a$ and that $P_0=P$.  Each $P_j$ clearly contains the PALF $L^{(0)}_j$ as a submanifold.

We apply the algorithm to each $P_j$. Let $m=(m_0,m_1,m_2)$ be a 3-tuple of non-negative integers satisfying $m_1\geq 1$, and let $l=(l_1, l_2)$ be a 2-tuple of non-negative integers. Put $p=(m_0,m_1,m_2,0,0,0, l_1,l_2)$. For an integer $i$, let $P^{(p)}_{j}$ and $P^{(p)}_{i,j}$ be the PALF's obtained by applying Step~\ref{step:modified PALF} to $P_j$, where we apply $R$-modifications so that the modified curves are independent of the index $j$. Put $P^{(p)}=P_{0}^{(p)}$ and let $\theta^{(p)}$ be the contact structure on $\partial P^{(p)}$ induced from the Stein structure on $P^{(p)}$. 

We may assume that the fiber $\mathbb{E}^{(p)}$ of the PALF's $P^{(p)}_{j}$ and $P^{(p)}_{i,j}$ is of genus one. The monodromy factorization of each $P^{(p)}_{i,j}$ is obtained from that of $P^{(p)}_{0,j}=P^{(p)}_j$ by replacing the subsequence $(\gamma_{1}, \beta, \gamma_{-1})$ with $(\gamma_{1}^{(i)}, \beta^{(i)}, \gamma_{-1}^{(i)})$, and the factorization of $P^{(p)}_{i,j}$ is obtained from that of $P^{(p)}_{i,0}$ by replacing the subsequence $(\rho_{1}, \beta, \rho_{-1})$ with $(\rho_{1}^{(j)}, b^{(j)}, \rho_{-1}^{(j)})$. Consequently, each $P^{(p)}_{i,j}$ is obtained by applying partial twists to $P^{(p)}$ along two subsurfaces $\mathbb{S}_\alpha$ and $\mathbb{S}_a$.

$P^{(p)}_{i,j}$'s satisfy the following properties. 

\begin{proposition}\label{sec:ex:prop:non-home_exotic}Fix non-negative integers $m_0,m_1,m_2, l_1,l_2$ with $m_1\geq 1$. Put $p=(m_0,m_1,m_2,0,0,0,l_1,l_2)$. Then the following hold. 

$(1)$ Fix an integer $j$. Then $P_{2i,j}^{(p)}$'s $(i\in \mathbb{Z})$ are pairwise homeomorphic Stein fillings of the same contact 3-manifold $(\partial P^{(p)}, \theta^{(p)})$. Moreover, infinitely many of them are pairwise non-diffeomorphic. 

$(2)$ Fix an integer $j$. Then $P_{2i-1,j}^{(p)}$'s $(i\in \mathbb{Z})$ are pairwise homeomorphic Stein fillings of the same contact 3-manifold $(\partial P^{(p)}, \theta^{(p)})$. Moreover, infinitely many of them are pairwise non-diffeomorphic. 

$(3)$ The support genus of the contact $3$-manifold $(\partial P^{(p)}, \theta^{(p)})$ is one. 

$(4)$ $\pi_1(P_{i,j}^{(p)})\cong \mathbb{Z}/j\mathbb{Z}$, for each integers $i,j$. Consequently, for any fixed integer $i$,  the Stein fillings $P_{i,j}^{(p)}$'s $(j\in \mathbb{Z})$ are pairwise non-homeomorphic. 

\end{proposition}
\begin{proof}We first calculate $\pi_1(P_j)$.  The handle decomposition of $P_j$ is obtained from that of $L^{(0)}_j$ by attaching 1-handles corresponding to $h^1_{\alpha_1}, h^1_{\alpha_2}, h^1_{\alpha_3}$ and 2-handles corresponding to $\gamma_{1}, \beta, \gamma_{-1}, \gamma_{-1}, \alpha_1,\alpha_2$. Clearly, each of these 1-handles is canceled with one of these 2-handles, and the rest of 2-handles do not affect the fundamental group of $L^{(0)}_j$. Since $\pi_1(L_j^{(0)})\cong\mathbb{Z}/j\mathbb{Z}$, we get $\pi_1(P_j)\cong \mathbb{Z}/j\mathbb{Z}$. 
Since $p$ satisfies the conditions of Step~\ref{step:condition:exotic}, and the fiber $\mathbb{E}^{(p)}$ is of genus one, the claims (1)--(4) immediately follow from Theorem~\ref{sec:algorithm:thm:exotic_PALF}. 
\end{proof}
\begin{proof}[Proof of Theorem~\ref{sec:intro:thm:non-homeo}]According to Proposition~\ref{sec:ex:prop:non-home_exotic}, for each fixed $p$ with the above conditions, $P_{i,j}^{(p)}$'s provide the desired Stein fillings for a contact 3-manifold. Though it is likely that we can obtain infinitely many such contact 3-manifolds by varying $p$, we take an alternative approach. Since there are many PALF's whose fiber is a surface of genus zero,  we obtain infinitely many such contact 3-manifolds by taking boundary connected sums of $P_{i,j}^{(p)}$ and such PALF's and applying Corollary~\ref{sec:algorithm:cor:boundary sum}.\end{proof}

Finally we raise the following problem, which naturally arises from Corollary~\ref{sec:algorithm:cor:boundary sum}. 
\begin{problem} Suppose that compact Stein 4-manifolds $X_1$ and $X_2$ are homeomorphic but non-diffeomorphic to each other, where their boundary contact structures are possibly non-contactomorphic. Does there exist a compact Stein 4-manifold $Y$ such that the boundary sums $X_1\natural Y$ and $X_2\natural Y$ are diffeomorphic to each other? 
\end{problem}


\end{document}